\documentclass[11pt]{amsart}

\usepackage{amsfonts, fancyhdr, qtree, amsmath, tipa, amssymb, hyperref, url, ifsym, amsthm, subfigure, enumitem, wasysym}

\newtheorem{thm}{Theorem}[section]
\newtheorem{lem}[thm]{Lemma}

\newtheorem{prop}[thm]{Proposition}

\newtheorem{cor}[thm]{Corollary}

\theoremstyle{definition}
\newtheorem{df}[thm]{Definition}
\theoremstyle{remark}
\newtheorem{rem}[thm]{Remark}

\numberwithin{equation}{section}

\title{Cyclic Sieving and Plethysm Coefficients}
\author{David B Rush}
\address{Department of Mathematics, Massachusetts Institute of Technology, 77 Massachusetts Ave, Cambridge, MA 02139}
\email{dbr@mit.edu}
\date{\today}

\begin{document}

\begin{abstract}

A combinatorial expression for the coefficient of the Schur function $s_{\lambda}$ in the expansion of the plethysm $p_{n/d}^d \circ s_{\mu}$ is given for all $d$ dividing $n$ for the cases in which $n=2$ or $\lambda$ is rectangular.  In these cases, the coefficient $\langle p_{n/d}^d \circ s_{\mu}, s_{\lambda} \rangle$ is shown to count, up to sign, the number of fixed points of an $\langle s_{\mu}^n, s_{\lambda} \rangle$-element set under the $d^{\text{th}}$ power of an order-$n$ cyclic action.  If $n=2$, the action is the Sch\"utzenberger involution on semistandard Young tableaux (also known as evacuation), and, if $\lambda$ is rectangular, the action is a certain power of Sch\"utzenberger and Shimozono's \textit{jeu-de-taquin} promotion.  

This work extends results of Stembridge and Rhoades linking fixed points of the Sch\"utzenberger actions to ribbon tableaux enumeration.  The conclusion for the case $n=2$ is equivalent to the domino tableaux rule of Carr\'e and Leclerc for discriminating between the symmetric and antisymmetric parts of the square of a Schur function.  

\end{abstract}

\maketitle

\section{Introduction}

Given an irreducible polynomial representation $V$ of $GL_m(\mathbb{C})$ with character $f(x_1, x_2, \ldots, x_m)$, the degree-$n$ power-sum plethysms \[(p_{n/d}^d \circ f)(x_1, x_2, \ldots, x_m) := f \left(x_1^{n/d}, x_2^{n/d}, \ldots, x_m^{n/d} \right)^d\] for $d$ dividing $n$ are a family of virtual characters that shed light onto the structure of the $n$-fold tensor power $V^{\otimes n}$.  For example, $p_1^2 \circ f = f^2$ is the character of $V^{\otimes 2}$ and $p_2 \circ f$ is the character of the Grothendieck group element $[\operatorname{Sym}^2(V)] - [\wedge^2(V)]$, so together they describe the decomposition of the tensor square into its symmetric and alternating components, \[V \otimes V = \operatorname{Sym}^2(V) \oplus \wedge^2(V).\]  

Both degree-$2$ power-sum plethysms admit combinatorial descriptions.  According to the celebrated Littlewood--Richardson rule, the coefficient of a Schur function $s_{\lambda}$ in the square of a Schur function $s_{\mu}$ is the number of Yamanouchi tableaux of shape $\lambda / \mu$ and content $\mu$.  For the coefficient of $s_{\lambda}$ in the plethysm $p_2 \circ s_{\mu}$, Carr\'e and Leclerc \cite{Carre} in 1995 presented an analogous rule facilitated by the exhibition of a collection of combinatorial objects they called Yamanouchi domino tableaux.  

In this article, we present as our first objective a consistent combinatorial interpretation for both plethysm coefficients relying only on Yamanouchi (ordinary) tableaux.  Considering the Sch\"utzenberger involution on a tableau set with cardinality given by the coefficient of $s_{\lambda}$ in $s_{\mu}^2$, we prove that the coefficient of $s_{\lambda}$ in $p_2 \circ s_{\mu}$ counts, up to sign, the number of tableaux fixed under the involution.  

Then we turn to our second objective --- extending the fixed-point approach to higher degree plethysm coefficients.  For all positive integers $n$, provided that $\lambda$ is rectangular, there is a natural order-$n$ cyclic action on the tableaux specified by the Littlewood--Richardson rule for the coefficient of $s_{\lambda}$ in $s_{\mu}^n$, and we prove that coefficient of $s_{\lambda}$ in $p_{n/d}^d \circ s_{\mu}$ counts (up to sign) the number of tableaux fixed under the $d^{\text{th}}$ power of the cyclic action.  This yields a consistent combinatorial interpretation for the coefficient of $s_{\lambda}$ in each degree-$n$ power-sum plethysm of $s_{\mu}$.  

The fixed-point approach is reminiscent of the \textit{cyclic sieving phenomenon} of Reiner, Stanton, and White \cite{Reiner}, a common occurrence in combinatorics in which the fixed points of the powers of a natural cyclic action on a finite set are enumerated by root-of-unity evaluations of an associated generating function.  Of course, our formulas do not constitute instances of the cyclic sieving phenomenon \textit{per se}.  Nonetheless, they jibe with the cyclic sieving paradigm: Not only is the Newton power sum $p_{n/d}^d$ a root-of-unity specialization of a Hall--Littlewood function, but a 1997 conjecture of Lascoux, Leclerc, and Thibon \cite{Lascoux} holds that the plethysm $p_{n/d}^d \circ s_{\mu}$ is itself a root-of-unity specialization of an LLT function.  

Thus, by matching plethysm coefficients to cardinalities of fixed-point sets of cyclic actions on tableaux, we contribute a complement to the Littlewood--Richardson rule that underscores the ubiquity of cyclic sieving in combinatorics and doubles as (heuristic) evidence for the longstanding Lascoux--Leclerc--Thibon conjecture.  

\subsection{Plethysms}
Let $\Lambda$ be the ring of symmetric functions over $\mathbb{Z}$ (cf. Macdonald \cite{Macdonald}).  For all $f, g \in \Lambda$, if $V$ and $W$ are polynomial representations of $GL_m(\mathbb{C})$ with characters $\chi_V = f(x_1, x_2, \ldots, x_m)$ and $\chi_W = g(x_1, x_2, \ldots, x_m)$, respectively, then $\chi_{V \oplus W} = (f+g)(x_1, x_2, \ldots, x_m)$ and $\chi_{V \otimes W} = (fg)(x_1, x_2, \ldots, x_m)$.  Plethysm is a binary operation on $\Lambda$ (so named by Littlewood \cite{Littlewood} in 1950) that is compatible with representation composition in the same sense that addition and multiplication correspond to representation direct sum and tensor product, respectively.  

To wit, if $\rho \colon GL_m(\mathbb{C}) \rightarrow GL_M(\mathbb{C})$ is a polynomial representation of $GL_m(\mathbb{C})$ with character $g(x_1, x_2, \ldots, x_m)$, and $\sigma \colon GL_M(\mathbb{C}) \rightarrow GL_N(\mathbb{C})$ is a polynomial representation of $GL_M(\mathbb{C})$ with character $f(x_1, x_2, \ldots, x_M)$, then the composition $\sigma \circ \rho \colon GL_m(\mathbb{C}) \rightarrow GL_N(\mathbb{C})$ is a polynomial representation of $GL_m(\mathbb{C})$ with character $(f \circ g)(x_1, x_2, \ldots, x_m)$, where $f \circ g \in \Lambda$ denotes the plethysm of $f$ and $g$.  A formal definition is given in section 2.  

We are herein concerned with plethysms of the form $p_{n/d}^d \circ s_{\mu}$, where $\mu$ is a partition, $s_{\mu}$ denotes the Schur function associated to $\mu$, $d$ divides $n$, and $p_{n/d}$ denotes the $(n/d)^{\text{th}}$ power-sum symmetric function, $x_1^{n/d} + x_2^{n/d} + \cdots$.  Defining an inner product $\langle \text{ }, \text{ } \rangle$ on $\Lambda$ by requiring that the Schur functions form an orthonormal basis, we obtain a convenient notation --- $\langle f, s_{\lambda} \rangle$ --- for the coefficient of $s_{\lambda}$ in the expansion of a symmetric function $f$ as a linear combination of Schur functions.  The main achievement in this article is a combinatorial description of the coefficients $\langle p_{n/d}^d \circ s_{\mu} ,  s_{\lambda} \rangle$ for the cases in which $n=2$ or $\lambda$ is rectangular.  

Let $\mu = (\mu_1, \mu_2, \ldots, \mu_m)$.  If $n=2$, the Littlewood--Richardson multiplicity $\langle s_{\mu}^n, s_{\lambda} \rangle$ is the number of semistandard Young tableaux of shape $\lambda$ and content $\overline{\mu} \mu := (\mu_m, \ldots, \mu_1, \mu_1, \ldots, \mu_m)$ for which the reading word is anti-Yamanouchi in $\lbrace 1, 2, \ldots, m \rbrace$ and Yamanouchi in $\lbrace m+1, m+2, \ldots, 2m \rbrace$.  The Sch\"utzenberger involution (also known as evacuation) on a semistandard tableau preserves the shape and reverses the content, so it gives an action on the tableaux of shape $\lambda$ and content $\overline{\mu} \mu$, which turns out to restrict to those tableaux with words satisfying the aforementioned Yamanouchi conditions (cf. Remark~\ref{evacdescends}).  

For the case in which $n$ may vary but $\lambda$ is rectangular, we treat the coefficient $\langle s_{\mu}^n, s_{\lambda} \rangle$ somewhat differently.  In general, the Littlewood--Richardson multiplicity $\langle s_{\mu}^n, s_{\lambda} \rangle$ is the number of semistandard Young tableaux of shape $\lambda$ and content $\mu^n := (\mu_1, \ldots, \mu_m, \mu_1, \ldots, \mu_m, \ldots, \mu_1, \ldots, \mu_m)$ for which the reading word is Yamanouchi in the alphabets $\lbrace km + 1, km + 2, \ldots, (k+1)m \rbrace$ for all $0 \leq k \leq n-1$.  On a semistandard tableau, \textit{jeu-de-taquin} promotion (also introduced by Sch\"utzenberger; cf. \cite{Schutz}) preserves the shape and permutes the content by the long cycle in $\mathfrak{S}_{mn}$, so $m$ iterations of promotion gives an action on the tableaux of shape $\lambda$ and content $\mu^n$.  If $\lambda$ is rectangular, this action has order $n$, and it, too, restricts to those tableaux with words satisfying the requisite Yamanouchi conditions (cf. Remark~\ref{promdescends}).  

We are at last poised to state our main results.  

\begin{thm} \label{mainevac}
Let $\textnormal{EYTab}(\lambda, \overline{\mu}{\mu})$ be the set of all semistandard tableaux of shape $\lambda$ and content $\overline{\mu} \mu$ with reading word anti-Yamanouchi in $\lbrace 1, 2, \ldots, m \rbrace$ and Yamanouchi in $\lbrace m+1, m+2, \ldots, 2m \rbrace$, and let $\xi$ act on $\textnormal{EYTab}(\lambda, \overline{\mu}{\mu})$ by the Sch\"utzenberger involution.  Then \[ | \lbrace T \in \textnormal{EYTab}(\lambda, \overline{\mu}{\mu}) : \xi(T) = T \rbrace | = \pm \left \langle p_2 \circ s_{\mu}, s_{\lambda} \right \rangle.\]  
\end{thm}

\begin{thm} \label{mainprom}
Let $\lambda$ be a rectangular partition, and let $\textnormal{PYTab}(\lambda, \mu^n)$ be the set of all semistandard tableaux of shape $\lambda$ and content $\mu^n$ with reading word Yamanouchi in the alphabets $\lbrace km+1, km+2, \ldots, (k+1)m \rbrace$ for all $0 \leq k \leq n-1$.  Let $j$ act on $\textnormal{PYTab}(\lambda, \mu^n)$ by $m$ iterations of \textnormal{jeu-de-taquin} promotion.  Then, for all positive integers $d$ dividing $n$, \[ | \lbrace T \in \textnormal{PYTab}(\lambda, \mu^n): j^{d}(T) = T \rbrace | = \pm \left \langle p_{n/d}^d \circ s_{\mu}, s_{\lambda} \right \rangle.\] 
\end{thm}

From Theorems 3.1 and 3.2 in Lascoux--Leclerc--Thibon \cite{Lascoux}, we see that the Hall--Littlewood symmetric function $Q'_{1^n}(q)$ specializes (up to sign) at $q = e^{\frac{2 \pi i \ell}{n}}$ to $p_{n/{\gcd(n,\ell)}}^{\gcd(n, \ell)}$.  Therefore, we may interpret Theorem~\ref{mainprom} as analogous to exhibiting an instance of the cyclic sieving phenomenon, and Theorem~\ref{mainevac} as analogous to exhibiting an instance of Stembridge's ``$q = -1$'' phenomenon (the progenitor of the cyclic sieving phenomenon for involutions; cf. \cite{Stembridge}).   

\begin{cor} \label{corevac}
Let $\xi$ act on $\textnormal{EYTab}(\lambda, \overline{\mu} {\mu})$ by the Sch\"utzenberger involution.  Then \[ | \lbrace T \in \textnormal{EYTab}(\lambda, \overline{\mu}{\mu}) : \xi(T) = T \rbrace | = \pm \left \langle Q'_{1^n}(-1), s_{\lambda} \right\rangle.\]  
\end{cor}

\begin{cor} \label{corprom}
Let $\lambda$ be a rectangular partition.  Let $j$ act on $\textnormal{PYTab}(\lambda, \mu^n)$ by $m$ iterations of \textnormal{jeu-de-taquin} promotion.  Then, for all integers $\ell$, \[| \lbrace T \in \textnormal{PYTab}(\lambda, \mu^n): j^{\ell}(T) = T \rbrace | = \pm \left \langle Q'_{1^n} \left(e^{\frac{2 \pi i \ell}{n}} \right) \circ s_{\mu}, s_{\lambda} \right \rangle.\]
\end{cor}

\begin{rem} \rm
The signs appearing in Theorems~\ref{mainevac} and \ref{mainprom} are predictable, and depend upon $\lambda$, $d$, and $n$ only.  Consult section 4, which contains the proofs of these theorems, for more details.  
\end{rem}

Theorem~\ref{mainevac} does not give the first combinatorial expression for the coefficient $\langle p_2 \circ s_{\mu}, s_{\lambda} \rangle$, but it distinguishes itself from the existing Carr\'e--Leclerc formula by its natural compatibility with the Littlewood--Richardson rule, and it is sufficiently robust that the techniques involved in its derivation are applicable to a whole class of plethysm coefficients with $n > 2$, addressed in Theorem~\ref{mainprom}, which is new in content and in form. 

In contrast, the Carr\'e--Leclerc rule has not been generalized to plethysms of degree higher than $2$, for the concept of Yamanouchi reading words has not been extended to $n$-ribbon tableaux for $n \geq 3$.  

Furthermore, the author has shown in unpublished work that a bijection of Berenstein and Kirillov \cite{Berenstein} between domino tableaux and tableaux stable under evacuation restricts to a bijection between those tableaux specified in the Carr\'e--Leclerc rule and in Theorem~\ref{mainevac}, respectively.  It follows that Theorem~\ref{mainevac} recovers the Carr\'e--Leclerc result.    

\subsection{Characters}
To prove Theorems~\ref{mainevac} and \ref{mainprom}, we turn to the theory of Lusztig canonical bases, which provides an algebraic setting for the Sch\"utzenberger actions evacuation and promotion.  In particular, we consider an irreducible representation of $GL_{mn}(\mathbb{C})$ for which there exists a basis indexed by the semistandard tableaux of shape $\lambda$ with entries in $\lbrace 1, 2, \ldots, mn \rbrace$ such that, if $n=2$, the long element $w_0 \in \mathfrak{S_{mn}} \hookrightarrow GL_{mn}$ permutes the basis elements (up to sign) by evacuation, and, if $\lambda$ is rectangular, the long cycle $c_{mn} \in \mathfrak{S_{mn}} \hookrightarrow GL_{mn}$ permutes the basis elements (up to sign) by promotion.  

With a suitable basis in hand, we proceed to compute the character $\chi$ of the representation at a particular element of $GL_{mn}$.  If $n=2$, we compute \[ \chi(w_0 \cdot \operatorname{diag}(x_1, x_2, \ldots, x_m, x_m, \ldots, x_2, x_1)),\] and, if $\lambda$ is rectangular, we compute \[\chi(c_{mn}^{md} \cdot \operatorname{diag}(y_1, y_2, \ldots, y_d, y_1, y_2, \ldots, y_d, \ldots, y_1, y_2, \ldots, y_d)),\] where the block $\operatorname{diag}(y_1, y_2, \ldots, y_d)$ occurs $n/d$ times along the main diagonal, and $y_i$ in turn represents the block $\operatorname{diag}(y_{i,1}, y_{i,2}, \ldots, y_{i,m})$ for all $1 \leq i \leq d$.

These character evaluations pick out the fixed points of the relevant order-$n$ cyclic actions.  Furthermore, they may be calculated by diagonalization of the indicated elements, for characters are class functions, and the values of the irreducible characters of $GL_{mn}$ at diagonal matrices are well known.  A careful inspection of the resulting formulas yields the desired identities.  

The relationship between $w_0$ and evacuation was first discovered by Berenstein and Zelevinsky \cite{BZ} in 1996, in the context of a basis dual to Lusztig's canonical basis.  In this article, we opt for an essentially equivalent basis constructed by Skandera \cite{Skandera}, which was used by Rhoades to detect the analogous relationship between $c_{mn}$ and promotion.  From the observations that $w_0$ and $c_{mn}$ lift the actions of evacuation and promotion, respectively, with respect to the dual canonical basis (or something like it), Stembridge \cite{Stembridge} and Rhoades \cite{Rhoades} deduced correspondences between fixed points of Sch\"utzenberger actions and ribbon tableaux, which inspired our results.    

Recall that an $r$-ribbon tableau of shape $\lambda$ is a tiling of the Young diagram of $\lambda$ by connected skew diagrams with $r$ boxes that contain no $2 \times 2$ squares (referred to as $r$-ribbons), each labeled by a positive integer entry.  (Thus, $1$-ribbon tableaux are ordinary tableaux, and $2$-ribbon tableaux are domino tableaux.)  If the entries of the $r$-ribbons are weakly increasing across each row and strictly increasing down each column, the $r$-ribbon tableau is called semistandard, by analogy with the definition of ordinary semistandard tableaux.  

\begin{thm}[Stembridge \cite{Stembridge}, Corollary 4.2] \label{Stemevac}
Let $\textnormal{Tab}(\lambda, \overline{\mu}{\mu})$ be the set of all semistandard tableaux of shape $\lambda$ and content $\overline{\mu} \mu$, and let $\xi$ act on $\textnormal{Tab}(\lambda, \overline{\mu}{\mu})$ by the Sch\"utzenberger involution.  Then \[| \lbrace T \in \textnormal{Tab}(\lambda, \overline{\mu}{\mu}) : \xi(T) = T \rbrace |\] is the number of domino tableaux of shape $\lambda$ and content $\mu$.  
\end{thm}

\begin{thm}[Rhoades \cite{Rhoades}, proof of Theorem 1.5] \label{Rhoadesprom}
Let $\lambda$ be a rectangular partition, and let $\textnormal{Tab}(\lambda, \mu^n)$ be the set of all semistandard tableaux of shape $\lambda$ and content $\mu^n$.  Let $j$ act on $\textnormal{Tab}(\lambda, \mu^n)$ by $m$ iterations of \textnormal{jeu-de-taquin} promotion.  Then, for all positive integers $d$ dividing $n$, \[| \lbrace T \in \textnormal{Tab}(\lambda, \mu^n) : j^d(T) = T \rbrace |\] is the number of $(n/d)$-ribbon tableaux of shape $\lambda$ and content $\mu^d$.  
\end{thm}

Unfortunately, the proofs of Theorems~\ref{Stemevac} and \ref{Rhoadesprom} cannot be directly adapted to obtain Theorems~\ref{mainevac} and \ref{mainprom}.  In order for the Yamanouchi restrictions on our tableaux sets to be made to appear in our character evaluations, an additional point of subtlety is needed.  We find relief in the insights offered us by the theory of Kashiwara crystals, which provides a framework not only for the study of the Sch\"utzenberger actions, but also for the reformulation of the Yamanouchi restrictions in terms of natural operators on semistandard tableaux.

\subsection{Crystals}
Let $\mathfrak{g}$ be a complex reductive Lie algebra with simply laced root system $\Phi$, and choose a set of simple roots $\lbrace \alpha_1, \alpha_2, \ldots, \alpha_t \rbrace$.  Let $P$ be the weight lattice of $\mathfrak{g}$.  A $\mathfrak{g}$-crystal is a finite set $B$ equipped with a weight map $\operatorname{wt} \colon B \rightarrow P$ and a pair of raising and lowering operators $e_i, f_i \colon B \rightarrow B \sqcup \lbrace 0 \rbrace$ for each $i$ that obey certain conditions.  Most notably, for all $b \in B$, if $e_i(b)$ is nonzero, then $\operatorname{wt}(e_i (b)) = \operatorname{wt}(b) + \alpha_i$, and if $f_i (b)$ is nonzero, then $\operatorname{wt}(f_i (b)) = \operatorname{wt}(b) - \alpha_i$.  

If $\mathfrak{g} = \mathfrak{gl}_{mn}$, then we may identify $P$ with $\mathbb{Z}^{mn}$ and choose for the simple roots the vectors $E_i - E_{i+1}$ for all $1 \leq i \leq mn-1$, where $E_i$ denotes the $i^\text{th}$ standard basis vector for all $1 \leq i \leq mn$.  If we take $B$ to be the set of semistandard tableaux of shape $\lambda$ with entries in $\lbrace 1, 2, \ldots, mn \rbrace$, with the weight of each tableau given by its content, there exists a suitable choice of operators $e_i$ and $f_i$ so that $B$ assumes the structure of a $\mathfrak{g}$-crystal.  Furthermore, the word of a tableau $b \in B$ is Yamanouchi with respect to the letters $i$ and $i+1$ if and only if $e_i$ vanishes at $b$, and anti-Yamanouchi with respect to $i$ and $i+1$ if and only if $f_i$ vanishes at $b$.  From this vantage point, it is easy to see that evacuation and promotion act on the tableaux sets indicated in our main theorems, for they (essentially) act on the set of crystal operators by conjugation.  

We close the introduction with an outline of the rest of the article.  In section 2, we provide the requisite background on tableaux and symmetric functions.  After reviewing the rudimentary definitions, we introduce plethysms, and we end with the observation of Lascoux, Leclerc, and Thibon \cite{Lascoux} that the classical relationship between tableaux and Schur functions evinces a more general relationship between ribbon tableaux and power-sum plethysms of Schur functions.  In section 3, we define Kashiwara crystals for a simply-laced complex reductive Lie algebra, before specializing to the $\mathfrak{gl}_{mn}$ setting, where we show how to assign a crystal structure to the pertinent tableaux sets.  We also examine the interactions between the Sch\"utzenberger actions and the raising and lowering crystal operators.  Because both of these sections are expository, we strive for brevity, but an earlier version of this work \cite{Rush} contains an expanded treatment.  

Finally, in section 4, we present proofs of Theorems~\ref{mainevac} and \ref{mainprom}.  Here the Berenstein--Zelevinsky \cite{BZ} and Rhoades \cite{Rhoades} lemmas underlying the proofs of Theorems~\ref{Stemevac} and \ref{Rhoadesprom} are summarized in the statement of Theorem~\ref{skanderabases}.  

\section{Tableaux and Symmetric Function Background}
In this section, we discuss the basic facts about Young tableaux and symmetric functions that are necessary for this article to be understood and placed in its proper context.\footnote{More comprehensive accounts of the fundamentals can be found in Stanley \cite{ec2}, Chapter 7 or Fulton \cite{Fulton}, Chapters 1-6 (of the two treatments, Fulton's is the more leisurely).  For more on the combinatorics of tableaux, see James--Kerber \cite{James}.  For more on plethysms, a reference \textit{par excellence} is Macdonald \cite{Macdonald} (but the presentation is considerably more abstract).}  We begin with the definition of a semistandard tableau.  

\begin{df}
Let $\kappa$ be a partition of $k$, and let $\eta = (\eta_1, \eta_2, \ldots, \eta_t)$ be a composition of $k$.  A \textit{semistandard Young tableau} of \textit{shape} $\kappa$ and \textit{content} $\eta$ is a filling of a Young diagram of shape $\kappa$ by positive integer entries, with one entry in each box, such that the entries are weakly increasing across each row and strictly increasing down each column, and such that the integer $i$ appears as an entry $\eta_i$ times for all $1 \leq i \leq t$.  A semistandard tableau of shape $\kappa$ and content $\eta$ is \textit{standard} if $\eta_i = 1$ for all $1 \leq i \leq k$.  
\end{df}

\begin{df}
Let $\iota$ and $\kappa$ be partitions such that $\iota_i \leq \kappa_i$ for all positive parts $\iota_i$ of $\iota$.  Let $\eta=(\eta_1, \eta_2, \ldots, \eta_t)$ be a composition of $|\kappa / \iota|$.  A \textit{semistandard skew tableau} of \textit{shape} $\kappa / \iota$ and \textit{content} $\eta$ is a filling of a skew diagram of shape $\kappa / \iota$ by positive integer entries, with one entry in each box, such that the entries are weakly increasing across each row and strictly increasing down each column, and such that the integer $i$ appears as an entry $\eta_i$ times for all $1 \leq i \leq t$.  
\end{df}

An \textit{$r$-ribbon} is a connected skew diagram of area $r$ that contains no $2 \times 2$ block of squares.  Given a partition $\kappa$ of $k$, we say that the \textit{$r$-core} of $\kappa$ is empty if there exists a tiling of a Young diagram of shape $\kappa$ by $r$-ribbons (cf. James--Kerber \cite{James}).  Such a tiling is referred to as an \textit{$r$-ribbon diagram} of \textit{shape} $\kappa$.  For the $r$-core of $\kappa$ to be empty, $r$ must divide $k$, but the converse is not true.  

\begin{df}
Let $\kappa$ be a partition of $k$, and suppose that the $r$-core of $\kappa$ is empty.  Let $\eta = (\eta_1, \eta_2, \ldots, \eta_t)$ be a composition of $\frac{k}{r}$.  A \textit{semistandard $r$-ribbon tableau} of \textit{shape} $\kappa$ and \textit{content} $\eta$ is a filling of an $r$-ribbon diagram of shape $\kappa$ by positive integer entries, with one entry in each $r$-ribbon, such that the entries are weakly increasing across each row and strictly increasing down each column, and such that the integer $i$ appears as an entry $\eta_i$ times for all $1 \leq i \leq t$.  
\end{df}

To each semistandard tableau, we may associate a word that contains all the entries of the tableau, called the reading word.  

\begin{df}
Given a semistandard tableau $T$, the \textit{reading word} of $T$, which we denote by $w(T)$, is the word obtained by reading the entries of $T$ from bottom to top in each column, beginning with the leftmost column, and ending with the rightmost column.  
\end{df}

If $T$ is a tableau of shape $\kappa \vdash k$ and content $\eta = (\eta_1, \eta_2, \ldots, \eta_t)$, then $w(T)$ is a word of length $k$ on the alphabet $\lbrace 1, 2, \ldots, t \rbrace$, and the integer $i$ appears as a letter $\eta_i$ times for all $1 \leq i \leq t$.  

The reading words of the tableaux specified in our main theorems, as well as in the Littlewood--Richardson rule, are characterized by properties named for Yamanouchi.  

\begin{df}
A word $w = w_1 w_2 \cdots w_k$ on the alphabet $\lbrace 1, 2, \ldots, t \rbrace$ is \textit{Yamanouchi} (\textit{anti-Yamanouchi}) with respect to the integers $i$ and $i+1$ if, when it is read backwards from the end to any letter, the resulting sequence $w_k, w_{k-1}, \ldots, w_j$ contains at least (at most) as many instances of $i$ as of $i+1$.  
\end{df}

\begin{df}
A word $w$ on the alphabet $\lbrace 1, 2, \ldots, t \rbrace$ is \textit{Yamanouchi} (\textit{anti-Yamanouchi}) in the subset $\lbrace i, i+1, \ldots, i' \rbrace$ if it is Yamanouchi (anti-Yamanouchi) with respect to each pair of consecutive integers in $\lbrace i, i+1, \ldots, i' \rbrace$.  
\end{df}

That concludes our litany of combinatorial definitions.  We turn to a brief overview of symmetric polynomials and symmetric functions.  

Let $\Lambda_m$ be the ring of symmetric polynomials in $m$ variables, and let $\Lambda$ be the ring of symmetric functions.  

\begin{df} \label{Schur}
Let $\kappa$ be a partition of a positive integer $k$.  For all compositions $\eta$ of $k$, we denote the monomial $x_1^{\eta_1} x_2^{\eta_2} \cdots$ by $x^{\eta}$, and, for all tableaux $T$ of shape $\kappa$ and content $\eta$, we write $x^T$ for $x^{\eta}$.  The \textit{Schur function} associated to $\kappa$ in the variables $x_1, x_2, \ldots$ is $s_{\kappa} := \sum_{T} x^{T}$, where the sum ranges over all tableaux $T$ of shape $\kappa$.  For all $m$, the \textit{Schur polynomial} associated to $\kappa$ in the $m$ variables $x_1, x_2, \ldots, x_m$ is $s_{\kappa}(x_1, x_2, \ldots, x_m)$.  
\end{df}

It is well known that the Schur polynomials in $m$ variables associated to partitions with at most $m$ positive parts form a basis for $\Lambda_m$, and that the Schur functions form a basis for $\Lambda$.  We define an inner product on $\Lambda$ by decreeing that the Schur basis be orthonormal.  

\begin{df}
Let $\langle \text{ }, \text{ } \rangle \colon \Lambda \times \Lambda \rightarrow \mathbb{Z}$ be an inner product given by $\langle s_{\iota}, s_{\kappa} \rangle = \delta_{\iota, \kappa}$ for all partitions $\iota$ and $\kappa$, where $\delta_{\iota, \kappa}$ denotes the Kronecker delta.  
\end{df}

Thus, if $f$ is a symmetric function, there exists a unique expression for $f$ as a linear combination of Schur functions, and the coefficients are given by the inner product: $f = \sum_{\kappa} \langle f, s_{\kappa} \rangle s_{\kappa}$.  We refer to the sum as the \textit{expansion} of $f$ on the Schur basis, and to the inner products $\langle f, s_{\kappa} \rangle$ as the \textit{expansion coefficients}.  

Symmetric polynomials in multiple variable sets may be expanded as sums of products of Schur polynomials in the constituent variable sets, with the expansion coefficients being uniquely determined because the products of Schur polynomials form a basis for the multiple-variable-set symmetric-polynomial ring.  These expansion coefficients are also given by symmetric function inner products (as a consequence of the self-biorthogonality of the Schur basis, which entails that the structure constants of multiplication and comultiplication in $\Lambda$ with respect to the Schur basis coincide).  

\begin{thm} \label{alphabetunique}
Let $d$ be a positive integer, and let \[ \lbrace a_{1,1}, a_{1,2}, \ldots, a_{1,m_1} \rbrace, \lbrace a_{2,1}, a_{2,2}, \ldots, a_{2,m_2} \rbrace, \ldots, \lbrace a_{d,1}, a_{d,2}, \ldots, a_{d,m_d} \rbrace \] be a collection of $d$ variable sets denoted by $a_1, a_2, \ldots, a_d$, respectively.  Let $\theta_1, \theta_2, \ldots, \theta_d$ range over all $d$-tuples of partitions.  Then the set of products $\lbrace s_{\theta_1}(a_1) s_{\theta_2}(a_2) \cdots s_{\theta_d}(a_d) \rbrace_{\theta_1, \theta_2, \ldots, \theta_d}$ constitutes a basis for the ring of symmetric polynomials in the variable sets $a_1, a_2, \ldots, a_d$, and, for all $f \in \Lambda$, \[f(a_1, a_2, \ldots, a_d) = \sum_{\theta_1, \theta_2, \ldots, \theta_d} \langle f, s_{\theta_1} s_{\theta_2} \cdots s_{\theta_d} \rangle s_{\theta_1}(a_1) s_{\theta_2}(a_2) \cdots s_{\theta_d}(a_d).\]
\end{thm}

\begin{proof}
The proof is by induction on $d$.  The base case $d = 2$ is proven in Chapter 7, Section 15 of Stanley \cite{ec2} (cf. Equation 7.66).  The inductive step is handled identically.  

(Stanley \cite{ec2} addresses the Hopf algebra interpretation of the result in Equation 7.67.)
\end{proof}

Finally, we come to the definition of plethysm, taken from Macdonald \cite{Macdonald}.  

\begin{df}\label{plethysm}
Let $f, g \in \Lambda$, and let $g$ be written as a sum of monomials, so that $g = \sum_{\eta} u_{\eta} x^{\eta}$, where $\eta$ ranges over an infinite set of compositions.  Let $\lbrace y_i \rbrace_{i=1}^{\infty}$ be a collection of proxy variables defined by $\prod_{i=1}^{\infty} (1+y_it) = \prod_{\eta} (1+x^{\eta}t)^{u_{\eta}}$.  The \textit{plethysm} of $f$ and $g$, which we denote by $f \circ g$, is the symmetric function $f(y_1, y_2, \ldots)$.  
\end{df}

\begin{rem}
Although the relation $\prod_{i=1}^{\infty} (1+y_it) = \prod_{\eta} (1+x^{\eta}t)^{u_{\eta}}$ only determines the elementary symmetric functions in the variables $y_1, y_2, \ldots$, it is well known that the ring of symmetric functions is generated as a $\mathbb{Z}$-algebra by the elementary symmetric functions, so the plethysm $f \circ g = f(y_1, y_2, \ldots)$ is indeed well-defined.  
\end{rem}

The following observation follows immediately from Definition~\ref{plethysm}.  

\begin{prop}
For all $f \in \Lambda$, the map $\Lambda \rightarrow \Lambda$ given by $g \mapsto g \circ f$ is a ring homomorphism.  
\end{prop}

There exists a family of symmetric functions for which the other choice of map given by plethysm, i.e. $g \mapsto f \circ g$, is also a ring homomorphism, for all $f$ belonging to this family.  

\begin{df}
For all positive integers $k$, the $k^\text{th}$ \textit{power-sum symmetric function} in the variables $x_1, x_2, \ldots$ is $p_k := x_1^k+x_2^k+ \cdots$.  
\end{df}

\begin{prop} \label{plethyswitch}
Let $g \in \Lambda$, and let $k$ be a positive integer.  Then $p_k \circ g = g \circ p_k = g(x_1^k, x_2^k, \ldots)$.  
\end{prop}

\begin{proof}
As in Definition~\ref{plethysm}, say that $g = \sum_{\eta} u_{\eta} x^{\eta}$.  Taking logarithms of each side in the equality $\prod_{i=1}^{\infty} (1+y_it) = \prod_{\eta} (1+x^{\eta}t)^{u_{\eta}}$, we obtain \[\sum_{i=1}^{\infty} \sum_{k=1}^{\infty} \frac{(-1)^{k-1}}{k} y_i^k t^k = \sum_{\eta} \left(u_{\eta} \sum_{k=1}^{\infty} \frac{(-1)^{k-1}}{k} (x^{\eta})^k t^k \right).\]  Interchanging the order of summation on each side yields \[p_k(y_1, y_2, \ldots) =  \sum_{\eta} u_{\eta} (x^{\eta})^k = g(x_1^k, x_2^k, \ldots).\]  Since $p_k \circ g = p_k(y_1, y_2, \ldots)$, it follows that $p_k \circ g = g(x_1^k, x_2^k, \ldots)$.  It should be clear that $g \circ p_k = g(x_1^k, x_2^k, \ldots)$ as well.  
\end{proof}

We may conclude that the map given by $g \mapsto p_k \circ g$ is a ring homomorphism for all positive integers $k$.  (In fact, $g \mapsto p_k \circ g$ is the degree-$k$ Adams operation in the $\lambda$-ring $\Lambda$).  We are therefore permitted to introduce an adjoint operator, which we denote by $\varphi_k$, given by $f \mapsto \sum_{\kappa} \langle f, p_k \circ s_{\kappa} \rangle s_{\kappa}$, where the sum ranges over all partitions $\kappa$.  Note that the equality $\langle \varphi_k(f), g \rangle = \langle f, p_k \circ g \rangle$ holds for all $f, g \in \Lambda$, which explains the nomenclature.  

Let $\kappa$ be a partition.  Just as the ordinary tableaux of shape $\kappa$ index the monomials of the Schur function $s_{\kappa}$, the $k$-ribbon tableaux of shape $\kappa$ index the monomials of the symmetric function $\varphi_k(s_{\kappa})$.  

\begin{thm}\label{lascoux}
Let $\kappa$ be a partition, and suppose that the $k$-core of $\kappa$ is empty.  For all compositions $\eta$ of $\frac{|\kappa|}{k}$, we denote the monomial $x_1^{\eta_1} x_2^{\eta_2} \cdots$ by $x^{\eta}$, and, for all $k$-ribbon tableaux $T$ of shape $\kappa$ and content $\eta$, we write $x^T$ for $x^{\eta}$.  Then $\varphi_k(s_{\kappa}) = \epsilon_k(\kappa) \sum_T x^T$, where the sum ranges over all $k$-ribbon tableaux of shape $\kappa$, and $\epsilon_k(\kappa)$ denotes the $k$-sign of $\kappa$.  
\end{thm}

\begin{proof}
Let $\left(\kappa^{(1)}, \kappa^{(2)}, \ldots, \kappa^{(k)}\right)$ be the $k$-quotient of $\kappa$.  Since the $k$-core of $\kappa$ is empty, it follows from a result of Littlewood \cite{Littlewoodmodular} that $\varphi_k(s_{\kappa}) = \epsilon_k(\kappa) s_{\kappa^{(1)}}s_{\kappa^{(2)}} \cdots s_{\kappa^{(k)}}$.  However, from Equation 24 in Lascoux--Leclerc--Thibon \cite{Lascoux}, we see that $s_{\kappa^{(1)}}s_{\kappa^{(2)}} \cdots s_{\kappa^{(k)}} = \sum_T x^T$, where the sum ranges over all $k$-ribbon tableaux of shape $\kappa$, as desired.  (This identity is an algebraic restatement of a bijection between $k$-tuples of tableaux of shapes $\left(\kappa^{(1)}, \kappa^{(2)}, \ldots, \kappa^{(k)}\right)$ and $k$-ribbon tableaux of shape $\kappa$, due in its original form to Stanton and White \cite{Stanton}.)   
\end{proof}

In view of Theorem~\ref{lascoux}, it is natural to ask if there is an analogue of the Littlewood--Richardson rule that describes the expansion coefficients of the power-sum plethysms $p_n \circ s_{\mu}$, or, more generally, $p_{n/d}^d \circ s_{\mu}$, for $d$ dividing $n$.  In the following sections, we see how this article provides a partial affirmative answer.    

\section{Crystal Structure on Tableaux}
For a complex reductive Lie algebra $\mathfrak{g}$, Kashiwara's $\mathfrak{g}$-crystals constitute a class of combinatorial models patterned on representations of $\mathfrak{g}$.  If the root system of $\mathfrak{g}$ is simply laced, there exists a set of axioms, enumerated by Stembridge \cite{Stemlocal}, that characterize the crystals arising directly from $\mathfrak{g}$-representations, which he calls regular.  Given a partition $\kappa$ with $s$ parts, the combinatorics of the weight space decomposition of the irreducible $\mathfrak{gl}_s$-representation with highest weight $\kappa$ is captured in the regular $\mathfrak{gl}_s$-crystal structure assigned to the semistandard tableaux of shape $\kappa$ with entries in $\lbrace 1, 2, \ldots, s \rbrace$.\footnote{In defining the crystal structure on the tableaux of a given shape, we require the number of parts of the shape to be well-defined, so we deviate from the convention of identifying compositions that differ only by terminal zeroes.  Note, however, that we may declare a composition to have $s$ parts so long as it has at most $s$ positive parts.} 

In this section, we review the crystal structure on tableaux, and we observe that it offers a natural setting for the consideration of evacuation and promotion, due to the relationship between these actions and the raising and lowering crystal operators.  We also see that the crystal perspective facilitates a recasting of the Yamanouchi conditions on tableaux reading words in terms of the vanishing or nonvanishing of the raising and lowering operators at the corresponding tableaux, viewed as crystal elements.  

We begin with the definition of a crystal, following Joseph \cite{Joseph}, and that of a regular crystal, following Stembridge \cite{Stemlocal}.  As the section progresses, some formal definitions are omitted, but more details may be found in Rush \cite{Rush} or other readily available sources.\footnote{For more on crystals, consult Joseph \cite{Joseph}.  For more on the crystal structure on tableaux, see Kashiwara--Nakashima \cite{Kashiwara}.  For more on \textit{jeu de taquin}, see Fulton \cite{Fulton}.  For more on evacuation and \textit{jeu-de-taquin} promotion, see Sch\"utzenberger \cite{Schutz} and Shimozono \cite{Shimozono}.  For more on promotion in crystals, see Bandlow--Schilling--Thi\'ery \cite{Bandlow}.} 

\begin{df} \label{whatiscrystal}
Let $\mathfrak{g}$ be a complex reductive Lie algebra with weight lattice $P$.  Let $\Delta = \lbrace \alpha_1, \alpha_2, \ldots, \alpha_t \rbrace$ be a choice of simple roots, and let $\lbrace \alpha_1^{\vee}, \alpha_2^{\vee}, \ldots, \alpha_t^{\vee} \rbrace$ be the corresponding simple coroots.  A \textit{$\mathfrak{g}$-crystal} is a finite set $B$ equipped with a map $\operatorname{wt} \colon B \rightarrow P$ and a pair of operators $e_i, f_i \colon B \rightarrow B \sqcup \lbrace 0 \rbrace$ for each $1 \leq i \leq t$ that satisfy the following conditions:  

\begin{enumerate}
\item[(i)] $\max \lbrace \ell : f_i^{\ell}(b) \neq 0 \rbrace - \max \lbrace \ell : e_i^{\ell} (b) \neq 0 \rbrace  = \langle \operatorname{wt}(b), \alpha_i^{\vee} \rangle$ for all $b \in B$;
\item[(ii)] $e_i (b) \neq 0$ implies $\operatorname{wt}(e_i (b)) = \operatorname{wt}(b) + \alpha_i$ and $f_i(b) \neq 0$ implies $\operatorname{wt}(f_i ( b)) = \operatorname{wt}(b) - \alpha_i$ for all $b \in B$;
\item[(iii)] $b' = e_i (b)$ if and only if $b = f_i (b')$ for all $b, b' \in B$.  
\end{enumerate}

We refer to $e_i$ as the \textit{raising operator} associated to $\alpha_i$, and we refer to $f_i$ as the \textit{lowering operator} associated to $\alpha_i$.  We write $\epsilon_i(b) := \max \lbrace \ell :  e_i^{\ell} (b) \neq 0 \rbrace$ for the maximum number of times the raising operator $e_i$ may be applied to $b$ without vanishing, and we write $\phi_i(b) := \max \lbrace \ell : f_i^{\ell} ( b) \neq 0 \rbrace$ for the maximum number of times the lowering operator $f_i$ may be applied to $b$ without vanishing.  We also define, for all $1 \leq i,j \leq t$:
\begin{itemize}
\item $\Delta_i \epsilon_j(b) := \epsilon_j(b)-\epsilon_j(e_i b)$;
\item $\Delta_i \phi_j(b) := \phi_j(e_i b) - \phi_j(b)$;
\item $\nabla_i \epsilon_j(b) := \epsilon_j (f_i b) - \epsilon_j(b)$;
\item $\nabla_i \phi_j(b) := \phi_j(b) - \phi_j(f_i b)$.  
\end{itemize}
\end{df}  

\begin{df} \label{whatismorphism}
Let $B$ and $B'$ be $\mathfrak{g}$-crystals.  A map of sets $\pi \colon B \rightarrow B'$ is a \textit{(strict) morphism of crystals} if $\operatorname{wt} \circ \pi = \operatorname{wt}$, and, for all $1 \leq i \leq t$, $\pi \circ e_i = e_i \circ \pi$ and $\pi \circ f_i = f_i \circ \pi$.  (Here we tacitly stipulate $\pi(0) := 0$.)  If $\pi$ is bijective, we say $\pi$ is an \textit{isomorphism}.  
\end{df}

\begin{df} \label{regular}
Let $\mathfrak{g}$ be simply laced.  A $\mathfrak{g}$-crystal $B$ is \textit{regular} if the Stembridge axioms on $\Delta_i \epsilon_j$, $\Delta_i \phi_j$, $\nabla_i \epsilon_j$, and $\nabla_i \phi_j$ hold (cf. Stembridge \cite{Stemlocal}, or Rush \cite{Rush} for a restatement in the notation of this article).  
\end{df}

\begin{df}
A $\mathfrak{g}$-crystal $B$ is \textit{connected} if the underlying graph --- in which elements of $B$ are vertices, and vertices $b$ and $b'$ are joined by an edge if there exists $i$ such that $e_i (b) = b'$ or $e_i (b') = b$ --- is connected.  Given a subset $C \subset B$, if the elements of $C$ are the vertices of a connected component of the underlying graph of $B$, then $C$, equipped with $\operatorname{wt}|_C$ and $e_i|_C, f_i|_C$ for all $i$, is a $\mathfrak{g}$-crystal, and we refer to $C$ as a \textit{connected component} of $B$.  
\end{df}

\begin{rem}
Regular, connected $\mathfrak{g}$-crystals should be viewed as depictions of irreducible representations of $\mathfrak{g}$.  
\end{rem}

\begin{df}
Let $B$ be a $\mathfrak{g}$-crystal.  An element $b \in B$ is a \textit{highest weight element} if $e_i$ vanishes at $b$ for all $i$.  If $b$ is the unique highest weight element of $B$, then $B$ is a \textit{highest weight crystal} of highest weight $\operatorname{wt}(b)$.  
\end{df}

This terminology is compatible with the natural partial order on $B$ given by the restriction of the root order on $P$ to the image of $\operatorname{wt}$ in the sense that, if $B$ is connected, the maximal elements under this partial order coincide precisely with the highest weight elements of $B$.  

If we restrict our attention to regular crystals, then saying a crystal is connected is equivalent to saying it is a highest weight crystal.  Furthermore, a regular, connected crystal $B$ with highest weight $b$ is uniquely characterized by the values $\phi_i(b)$ for $1 \leq i \leq t$.  

\begin{prop} \label{regconequiv}
Let $B$ be a regular, connected $\mathfrak{g}$-crystal.  Then $B$ is a highest weight crystal.  
\end{prop}

\begin{prop} \label{regconiso}
Let $B$ and $B'$ be regular, connected $\mathfrak{g}$-crystals with highest weight elements $b$ and $b'$, respectively.  If $\operatorname{wt}(b) = \operatorname{wt}(b')$ and $\phi_i(b) = \phi_i(b')$ for all $1 \leq i \leq t$, then $B$ and $B'$ are isomorphic.  
\end{prop}

\begin{proof}
Propositions~\ref{regconequiv} and ~\ref{regconiso} are proved in Stembridge \cite{Stemlocal} under the assumption that $\mathfrak{g}$ is semisimple (in which case the hypothesis $\operatorname{wt}(b) = \operatorname{wt}(b')$ in Proposition~\ref{regconiso} is unnecessary).

To extend these results to our setting, let $\mathfrak{g}$ be reductive with Cartan subalgebra $\mathfrak{h}$ and weight lattice $P \subset \mathfrak{h}^*$, and let $\mathfrak{g} = \mathfrak{s} \oplus \mathfrak{z}(\mathfrak{g})$ be a Levi decomposition of $\mathfrak{g}$ such that $\mathfrak{t} := \mathfrak{h} \cap \mathfrak{s}$ is a Cartan subalgebra of $\mathfrak{s}$.  Let $Q \subset \mathfrak{t}^*$ be the weight lattice of $\mathfrak{s}$.  

A $\mathfrak{g}$-crystal $B$ inherits the structure of an $\mathfrak{s}$-crystal via the map $P \rightarrow Q$ obtained from the projection $\mathfrak{h}^* \rightarrow \mathfrak{t}^*$.  Furthermore, $B$ is regular and connected as an $\mathfrak{s}$-crystal if and only if it is regular and connected as a $\mathfrak{g}$-crystal.  

Since $\mathfrak{s}$ is semisimple, Proposition~\ref{regconequiv} follows immediately.  For Proposition~\ref{regconiso}, the corresponding statement governing $\mathfrak{s}$-crystals entails the existence of a bijective map $\pi \colon B \rightarrow B'$ such that $\pi \circ e_i = e_i \circ \pi$ and $\pi \circ f_i = f_i \circ \pi$ for all $i$.  To see that $\operatorname{wt}(b) = \operatorname{wt}(b')$ implies $\operatorname{wt} \circ \pi = \operatorname{wt}$, note that an element $a \in B$ may be expressed in the form \[a = f_{i_k} \cdots f_{i_2} f_{i_1} (b)\] for $i_1, i_2, \ldots, i_k \in \lbrace 1, 2, \ldots, t \rbrace$, so \[\pi(a) = f_{i_k} \cdots f_{i_2} f_{i_1}(b'),\] and $\operatorname{wt}(\pi(a)) = \operatorname{wt}(a)$.  
\end{proof}

Specializing to the case $\mathfrak{g} = \mathfrak{gl}_s$, we take as our Cartan subalgebra $\mathfrak{h}$ the subspace of diagonal matrices, and we identify $\mathfrak{h}^*$ with the space $\mathbb{C}^s$, where $E_i$ denotes the $i^{\text{th}}$ standard basis vector for all $1 \leq i \leq s$.  Then the weight lattice $P$ is generated over $\mathbb{Z}$ by $\lbrace E_1, E_2, \ldots, E_s \rbrace$, and we choose the set of simple roots $\lbrace \alpha_1, \alpha_2, \ldots, \alpha_{s-1} \rbrace$ in accordance with the rule $\alpha_i := E_i - E_{i+1}$ for all $1 \leq i \leq s-1$.  

To each partition $\kappa$ with $s$ parts, we impose a $\mathfrak{gl}_s$-crystal structure on the tableaux of shape $\kappa$ with entries in $\lbrace 1, 2, \ldots, s \rbrace$ such that the highest weight is $\kappa$.  To do so, we begin by defining a $\mathfrak{gl}_s$-crystal structure on the skew tableaux of shape $\kappa / \iota$, and then we reduce to the case in which the partition $\iota$ is empty.  

\begin{prop}[Kashiwara--Nakashima \cite{Kashiwara}] \label{tabstructure}
Let $\kappa$ and $\iota$ be partitions, each with $s$ parts, such that $\iota_i \leq \kappa_i$ for all positive parts $\iota_i$ of $\iota$.  Let $B_{\kappa / \iota}$ be the set of semistandard skew tableaux of shape $\kappa / \iota$ with entries in $\lbrace 1, 2, \ldots, s \rbrace$.  

Let the maps
\begin{align*}
& \operatorname{wt} \colon B_{\kappa / \iota} \rightarrow \mathbb{Z}^s\\
& h_{i,j}, k_{i,j} \colon B_{\kappa / \iota} \rightarrow \mathbb{Z} \\
& e_i, f_i \colon B_{\kappa / \iota} \rightarrow B_{\kappa / \iota} \sqcup \lbrace 0 \rbrace \\
\end{align*}
be given for all $1 \leq i \leq s-1$ and $j \in \mathbb{N}$ by stipulating, for all $T \in B_{\kappa / \iota}$: 
\begin{itemize}
\item $\operatorname{wt}(T)$ to be the content of $T$;
\item $h_{i,j}(T)$ to be the number of occurrences of $i+1$ in the $j^{\text{th}}$ column of $T$ or to the right minus the number of occurrences of $i$ in the $j^{\text{th}}$ column of $T$ or to the right;
\item $k_{i,j}(T)$ to be the number of occurrences of $i$ in the $j^{\text{th}}$ column of $T$ or to the left minus the number of occurrences of $i+1$ in the $j^{\text{th}}$ column or to the left;
\item $e_i(T)$ to be the skew tableau with an $i$ in place of an $i+1$ in the rightmost column for which $h_{i,j}(T)$ is maximal and positive if such a column exists, and 0 otherwise;
\item $f_i(T)$ to be the skew tableau with an $i+1$ in place of an $i$ in the leftmost column for which $k_{i,j}(T)$ is maximal and positive if such a column exists, and 0 otherwise.  
\end{itemize}
Then the set $B_{\kappa / \iota}$ equipped with the map $\operatorname{wt}$ and the operators $e_i, f_i$ for all $1 \leq i \leq s-1$ is a $\mathfrak{gl}_s$-crystal.  
\end{prop}

\begin{prop} \label{hwcrystal}
Let $\kappa$ be a partition with $s$ parts.  The $\mathfrak{gl}_s$-crystal $B_{\kappa} := B_{\kappa / \varnothing}$ is a regular, connected crystal of highest weight $\kappa$.  The highest weight element is the unique tableau of shape $\kappa$ and content $\kappa$.  
\end{prop}

\begin{proof}
It is proven that $B_{\kappa}$ is regular and connected as an $\mathfrak{sl}_s$-crystal in Stembridge \cite{Stemlocal}.  It follows that $B_{\kappa}$ is regular and connected as a $\mathfrak{gl}_s$-crystal, so, by Proposition~\ref{regconequiv}, it is a highest weight crystal.  To conclude, note that the unique tableau of shape $\kappa$ and content $\kappa$ is a highest weight element.   
\end{proof}

Fundamental to the study of skew tableaux is a procedure devised by Sch\"utzenberger for transforming a skew tableau into a tableau of left-justified shape, which we refer to as its rectification.  Given a skew tableau $T$ of shape $\kappa / \iota$, \textit{jeu de taquin} calls for the boxes in the Young diagram of shape $\iota$ to be relocated one at a time from the northwest to the southeast of $T$ via a sequence of successive slides.  These \textit{jeu-de-taquin} slides commute with the raising and lowering operators, so we consider \textit{jeu de taquin} to respect the crystal structure on tableaux.  

\begin{prop}[Bandlow--Schilling--Thi\'ery \cite{Bandlow}, Remarks 3.3] \label{jdtcommutes}
Let $\kappa$ and $\iota$ be nonempty partitions such that $\iota_i \leq \kappa_i$ for all positive parts $\iota_i$ of $\iota$.  Let $C$ be a box in the Young diagram of shape $\iota$ for which neither the box below nor the box to the right are in $\iota$.  For all semistandard skew tableaux $T$ of shape $\kappa / \iota$, let $\operatorname{jdt}(T)$ be the result of a \textnormal{jeu-de-taquin} slide on $T$ starting from $C$, and set $\operatorname{jdt}(0) := 0$.  Then $e_i (\operatorname{jdt}(T)) = \operatorname{jdt}(e_i (T))$ and $f_i (\operatorname{jdt}(T)) = \operatorname{jdt}(f_i(T))$ for all $T \in B_{\kappa / \iota}$ and $1 \leq i \leq s-1$.  
\end{prop}

\begin{cor} \label{rectcommutes}
Let $\kappa$ and $\iota$ be partitions such that $\iota_i \leq \kappa_i$ for all positive parts $\iota_i$ of $\iota$.  Let $T \in B_{\kappa / \iota}$, and let $\operatorname{Rect}(T)$ be the rectification of $T$.  Then $\epsilon_i(T) = \epsilon_i(\operatorname{Rect}(T))$ and $\phi_i(T) = \phi_i(\operatorname{Rect}(T))$ for all $1 \leq i \leq s-1$.  
\end{cor}

The natural action of $\mathfrak{S}_s$ on compositions with $s$ parts given by $w \cdot (\eta_1, \eta_2, \ldots, \eta_s) := \left(\eta_{w^{-1}(1)}, \eta_{w^{-1}(2)}, \ldots, \eta_{w^{-1}(s)}\right)$ yields an $\mathfrak{S}_s$-action on the contents (and therefore the weights) of tableaux with entries in $\lbrace 1, 2, \ldots, s \rbrace$.  With \textit{jeu de taquin} at his disposal, Sch\"utzenberger \cite{Schutz} introduced a pair of cyclic actions on tableaux that lift permutations on their contents.  

\textit{Jeu-de-taquin} promotion, generalized to our setting by Shimozono \cite{Shimozono}, may be thought of as first turning the $1$'s in a tableau into $2$'s, the $2$'s into $3$'s, etc., and the $s$'s into $1$'s, followed by rearranging the entries via \textit{jeu de taquin} so that the result remains a valid tableau.

The Sch\"utzenberger involution, also referred to as evacuation, may be thought of as turning the $1$'s in a tableau into $s$'s, the $2$'s into $s-1$'s, etc., via a concatenation of $s-1$ promotions, corresponding to the canonical decomposition of the long element in $\mathfrak{S}_s$ into a product of $s-1$ cycles with one descent each, viz., $w_0 = (12 \cdots s) \cdots (123) (12)$.  

Because promotion and evacuation are derived from \textit{jeu de taquin}, it should be no surprise that they inherit compatibility with the raising and lowering crystal operators.  

\begin{prop} [Bandlow--Schilling--Thi\'ery \cite{Bandlow}, Proposition 3.2] \label{prominteracts}
Let $\kappa$ be a partition with $s$ parts, and let $\operatorname{pr} \colon B_{\kappa} \rightarrow B_{\kappa}$ be \textnormal{jeu-de-taquin} promotion.  Set $\operatorname{pr}(0) := 0$.  Then, for all $T \in B_{\kappa}$:
\begin{enumerate}
\item[(i)] $\operatorname{wt}(\operatorname{pr}(T)) = c_s \cdot \operatorname{wt}(T)$;
\item[(ii)] $\operatorname{pr}(e_i(T)) = e_{i+1} ( \operatorname{pr}(T))$ and $\operatorname{pr}(f_i(T)) = f_{i+1} ( \operatorname{pr}(T))$ for all $1 \leq i \leq s-2$.  
\end{enumerate}
\end{prop}

\begin{prop} [Lascoux--Leclerc--Thibon \cite{Lascouxyz}, Section 3] \label{evacinteracts}
Let $\kappa$ be a partition with $s$ parts, and let $\xi \colon B_{\kappa} \rightarrow B_{\kappa}$ be the Sch\"utzenberger involution.  Set $\xi(0) := 0$.  Then, for all $T \in B_{\kappa}$:
\begin{enumerate}
\item[(i)] $\operatorname{wt}(\xi(T)) = w_0 \cdot \operatorname{wt}(T)$;
\item[(ii)] $\xi(e_i (T)) = f_{s-i} (\xi(T))$ and $\xi(f_i (T)) = e_{s-i} ( \xi(T))$ for all $1 \leq i \leq s-1$.  
\end{enumerate}
\end{prop}

The properties in Proposition~\ref{prominteracts} and ~\ref{evacinteracts} completely characterize promotion and evacuation.  

\begin{thm}[Bandlow--Schilling--Thi\'ery \cite{Bandlow}, Proposition 3.2] \label{uniqueweakprom}
Let $\kappa$ be a partition with $s$ parts, and let $\operatorname{pr} \colon B_{\kappa} \rightarrow B_{\kappa}$ be \textnormal{jeu-de-taquin} promotion.  If an action $\gamma \colon B_{\kappa} \rightarrow B_{\kappa}$ satisfies the properties of promotion delineated in Proposition~\ref{prominteracts}, then $\gamma$ and $\operatorname{pr}$ coincide.  
\end{thm}

\begin{thm} [Henriques--Kamnitzer \cite{Henriques}, Section 5.D]
Let $\kappa$ be a partition with $s$ parts, and let $\xi \colon B_{\kappa} \rightarrow B_{\kappa}$ be the Sch\"utzenberger involution.  If an action $\gamma \colon B_{\kappa} \rightarrow B_{\kappa}$ satisfies the properties of evacuation delineated in Proposition~\ref{evacinteracts}, then $\gamma$ and $\xi$ coincide.
\end{thm} 

The following theorem reveals why we restrict our attention to rectangular partitions in the statement of Theorem~\ref{mainprom}.  

\begin{df}
A partition $\kappa$ is \textit{rectangular} if all its positive parts are equal.  
\end{df}

\begin{thm}[Bandlow--Schilling--Thi\'ery \cite{Bandlow}, Proposition 3.2] \label{uniqueprom}
Let $\kappa$ be a partition with $s$ parts, and let $\operatorname{pr} \colon B_{\kappa} \rightarrow B_{\kappa}$ be \textnormal{jeu-de-taquin} promotion.  Then $\operatorname{pr}^s$ acts as the identity if and only if $\kappa$ is rectangular.  
\end{thm}

\begin{rem}
Together, Theorems~\ref{uniqueweakprom} and \ref{uniqueprom} testify at once to the potency of our techniques for investigating rectangular tableaux and to the difficulty in extending them beyond the rectangular setting.  Indeed, to address the general case in accordance with the cyclic sieving paradigm, we require a cyclic action of order $s$ on $B_{\kappa}$.  Theorem~\ref{uniqueweakprom} tells us that the only cyclic action compatible with the crystal operators (at least in the way we understand compatibility) is \textit{jeu-de-taquin} promotion, but, by Theorem~\ref{uniqueprom}, promotion is of the correct order if and only if $\kappa$ is rectangular.  
\end{rem}

Finally, as promised, we reinterpret the Yamanouchi conditions on reading words as vanishing conditions on crystal operators.  We end by noting that the Yamanouchi conditions completely characterize the highest and lowest weight elements of the $\mathfrak{gl}_s$-crystals comprising semistandard tableaux defined in Proposition~\ref{tabstructure}.  Because the following propositions are essentially self-evident, we omit the proofs.  

\begin{prop}
Let $\kappa$ be a partition with $s$ parts, and let $T$ be a tableau of shape $\kappa$.  For all $1 \leq i \leq s-1$, the word of $T$ is Yamanouchi (anti-Yamanouchi) with respect to the integers $i$ and $i+1$ if and only if the raising operator $e_i$ (lowering operator $f_i$) vanishes at $T$.  
\end{prop}

\begin{prop} \label{yamreinterp}
Let $\kappa$ be a partition with $s$ parts, and let $T$ be a tableau of shape $\kappa$.  For all $1 \leq i < i' \leq s-1$, the word of $T$ is Yamanouchi (anti-Yamanouchi) in the subset $\lbrace i, i+1, \ldots, i' \rbrace$ if and only if the raising operators  $e_i, e_{i+1}, \ldots, e_{i'-1}$ (lowering operators $f_i, f_{i+1}, \ldots, f_{i'-1}$) all vanish at $T$.    
\end{prop}

\begin{prop} \label{highweightlemma}
Let $T$ be a tableau with entries in $\lbrace 1, 2, \ldots, m \rbrace$ and Yamanouchi reading word.  Then $T$ is of shape $\mu$ if and only if $T$ is of content $\mu$.  
\end{prop}

\begin{prop} \label{lowweightlemma}
Let $T$ be a tableau with entries in $\lbrace 1, 2, \ldots, m \rbrace$ and anti-Yamanouchi reading word.  Then $T$ is of shape $\mu$ if and only if $T$ is of content $\overline{\mu}$.  
\end{prop}

\section{Proofs of Theorems~\ref{mainevac} and \ref{mainprom}}

In this section, we prove our main theorems.  We start with an overview of the basis of Kazhdan--Lusztig immanants constructed by Skandera \cite{Skandera} for the dual of an irreducible polynomial representation of $GL_s(\mathbb{C})$.  For $\kappa$ a partition with $s$ parts, we note that the action of the long element $w_0 \in \mathfrak{S}_s \subset GL_s(\mathbb{C})$ on the immanants associated to the tableaux of shape $\kappa$ lifts (up to sign) the Sch\"utzenberger involution on the $\mathfrak{sl}_s$-crystal $B_{\kappa}$, and, analogously, that the action of the long cycle $c_s \in \mathfrak{S}_s \subset GL_s(\mathbb{C})$ on immanants lifts (up to sign) \textit{jeu-de-taquin} promotion if $\kappa$ is rectangular.  (The claim for promotion is due to Rhoades \cite{Rhoades}; the author derived in \cite{Rush} the corresponding claim for evacuation from lemmas of Berenstein--Zelevinsky \cite{Berenstein} and Stembridge \cite{Stembridge} by mimicking Rhoades's argument.)  Setting $s := mn$, we then derive the desired conclusions from character computations, drawing upon the background developed in the two preceding sections.  Some familiarity with the character theory of $GL_s(\mathbb{C})$ is assumed.\footnote{The algebraic tools used in this section are developed in greater depth in the original version of this work \cite{Rush}, but even the discussion there is necessarily abbreviated.  For more details on the irreducible polynomial characters of $GL_s(\mathbb{C})$, consult Fulton \cite{Fulton}, Chapter 8.  For more about the Kazhdan--Lusztig basis, see the original paper by Kazhdan and Lusztig \cite{Kazhdan}, or Bj\"orner--Brenti \cite{Bjorner} for an expository account.  The crucial facts concerning the Skandera bases may be found in Rhoades--Skandera \cite{Rhoadess} and Skandera \cite{Skandera}.  The entire section is informed by Rhoades's article ``Cyclic sieving, promotion, and representation theory'' \cite{Rhoades}, to which a considerable intellectual debt is owed and appreciated.}

\begin{thm}[Rhoades \cite{Rhoades}, Rhoades--Skandera \cite{Rhoadess}, Rush \cite{Rush}] \label{skanderabases}
Let $\kappa$ be a partition of $t$ with $s$ parts, and let $V_{\kappa, s}$ be the dual of the irreducible polynomial $GL_s(\mathbb{C})$-representation with highest weight $\kappa$.  For all compositions $\eta$ of $t$ with $s$ parts and semistandard tableaux $U$ of shape $\kappa$ and content $\eta$, let $I_{\eta}(U) \in V_{\kappa, s}$ be the Kazhdan--Lusztig immanant associated to $\eta$ and $U$.\footnote{To construct the \textit{Kazhdan--Lusztig immanant} associated to $\eta$ and $U$, we start with a permutation $w \in \mathfrak{S}_t$, determined by $U$ via the Robinson--Schensted--Knuth algorithm, and we build the polynomial \[\operatorname{Imm}_w(x) := \sum_{v \geq w} (-1)^{\ell(v) - \ell(w)} P_{w_0 v, w_0 w}(1) x_{1,v(1)} x_{2, v(2)} \cdots x_{t, v(t)} \in \operatorname{Sym}((\mathbb{C}^t)^* \otimes (\mathbb{C}^t)^*), \] where $P_{w_0 v, w_0 w}(q)$ is the Kazhdan--Lusztig polynomial associated to the (ordered) pair $w_0 v, w_0 w$.  The composition $\eta$ determines a map $\lbrace 1, 2, \ldots, t \rbrace \rightarrow \lbrace 1, 2 \ldots, s \rbrace$, which induces a map $((\mathbb{C}^t)^* \otimes (\mathbb{C}^t)^*) \rightarrow ((\mathbb{C}^s)^* \otimes (\mathbb{C}^t)^*)$, and we denote the image of $\operatorname{Imm}_w(x)$ by $\operatorname{Imm}_w(x_{\eta})$.  Then $I_{\eta}(U)$ is in turn the image of $\operatorname{Imm}_w(x_{\eta})$ in $V_{\kappa, s}$.  \textit{Caveat lector}: We refer to this image as $I_{\eta}(U')$ in the notation of Rush \cite{Rush}, where $U'$ denotes the row-strict tableau conjugate to $U$.}

Set \[I_{\eta} := \lbrace I_{\eta}(U) : U \text{ is a semistandard tableau of shape $\kappa$ and content $\eta$} \rbrace.\]  Then the following claims hold.  
\begin{enumerate}
\item[(i)] The set $\bigcup_{\eta} I_{\eta}$, where $\eta$ ranges over all compositions of $t$ with $s$ parts, constitutes a basis for $V_{\kappa,s}$.  
\item[(ii)] For all compositions $\eta$ of $t$ with $s$ parts, the set $I_{\eta}$ constitutes a basis for the weight space of $V_{\kappa,s}$ corresponding to the weight $-\eta$, which we denote by $V_{\kappa,s,\eta}$.    
\item[(iii)]
Let $w_0$ be the long element in $\mathfrak{S}_s$, and let $\xi$ be the Sch\"utzenberger involution.  Let $a$ be the number of positive parts of $\kappa$, and write $v(\kappa)$ for the sum $\sum_{i=1}^a (i-1) \kappa_i$.  Then \[w_0 \cdot I_{\eta}(U) = (-1)^{v(\kappa)} \cdot I_{w_0 \cdot \eta}(\xi(U)).\]
\item[(iv)]
Let $c_s$ be the long cycle in $\mathfrak{S}_s$, and let $\operatorname{pr}$ be \textnormal{jeu-de-taquin} promotion.  Let $a$ be the number of positive parts of $\kappa$.  If $\kappa$ is rectangular, then \[c_s \cdot I_{\eta}(U) = (-1)^{\eta_s(a-1)} \cdot I_{c_s \cdot \eta}(\operatorname{pr}(U)).\]

\end{enumerate}
\end{thm}

\subsection{Proof of Theorem~\ref{mainevac}}

Let $\lambda$ be a partition with $2m$ parts, and let $\mu$ be a partition of $|\lambda|/2$ with $m$ parts such that $\mu_i \leq \lambda_i$ for all positive parts $\mu_i$ of $\mu$.  Denote the composition $(\mu_m, \mu_{m-1}, \ldots, \mu_1, \mu_1, \mu_2, \ldots, \mu_m)$ by $\overline{\mu}{\mu}$.  Write $\textnormal{Tab}(\lambda, \overline{\mu}{\mu})$ for the set of semistandard tableaux of shape $\lambda$ and content $\overline{\mu} \mu$, and $\textnormal{EYTab}(\lambda, \overline{\mu}{\mu})$ for the subset of $\textnormal{Tab}(\lambda, \overline{\mu}{\mu})$ consisting of those tableaux with reading word anti-Yamanouchi in $\lbrace 1, 2, \ldots, m \rbrace$ and Yamanouchi in $\lbrace m+1, m+2, \ldots, 2m \rbrace$.  

\begin{rem}\label{departure}
Again we depart from the standard convention of identifying compositions that differ only by terminal zeroes, but, given partitions $\lambda$ and $\mu$, we may choose $m$ so that $\lambda$ and $\mu$ have at most $2m$ and $m$ positive parts, respectively, and declare $\lambda$ and $\mu$ to have $2m$ and $m$ parts, respectively.  It should be clear that the choice of $m$ does \textit{not} affect the cardinalities of the tableaux sets in question.
\end{rem}

Let $B_{\lambda}$ be the set of semistandard tableaux of shape $\lambda$, endowed with a $\mathfrak{gl}_{2m}$-crystal structure in accordance with Proposition~\ref{tabstructure}.  The key to our proof is the assignment of a $(\mathfrak{gl}_m \oplus \mathfrak{gl}_m)$-crystal structure to $B_{\lambda}$ that allows us to inspect the action of $\xi$ on its connected components.  This provides a combinatorial model for the decomposition into irreducible components of the restriction to $GL_m(\mathbb{C}) \times GL_m(\mathbb{C})$ of the irreducible $GL_{2m}(\mathbb{C})$-representation with highest weight $\lambda$, which underlies our character evaluation.  

Recall that we chose $\lbrace E_1 - E_2, E_2 - E_3, \ldots, E_{2m-1} - E_{2m} \rbrace$ as the set of simple roots for $\mathfrak{gl}_{2m}$.  Here we choose $\lbrace E_2-E_1, E_3-E_2, \ldots, E_m-E_{m-1}, E_{m+1} -E_{m+2}, E_{m+2} - E_{m+3}, \ldots, E_{2m-1}-E_{2m} \rbrace$ as the set of simple roots for $\mathfrak{gl}_m \oplus \mathfrak{gl}_m$.    

\begin{prop} \label{evaccrysrestrict}
The set $B_{\lambda}$ equipped with the map $\operatorname{wt}$, the set of raising operators $\lbrace f_1, f_2, \ldots, f_{m-1}, e_{m+1}, e_{m+2}, \ldots, e_{2m-1} \rbrace$, and the set of lowering operators $\lbrace e_1, e_2, \ldots, e_{m-1}, f_{m+1}, f_{m+2}, \ldots, f_{2m-1} \rbrace$, is a regular $(\mathfrak{gl}_m \oplus \mathfrak{gl}_m)$-crystal. 
\end{prop}

\begin{proof}
It is a simple matter to verify that the conditions of Definition~\ref{whatiscrystal} hold for $\mathfrak{g} = \mathfrak{gl}_m \oplus \mathfrak{gl}_m$ with the indicated choice of simple roots.  Hence $B_{\lambda}$ is a $(\mathfrak{gl}_m \oplus \mathfrak{gl}_m)$-crystal.  Furthermore, drawing any two operators from distinct sets among \[\lbrace e_1, e_2, \ldots, e_{m-1}, f_1, f_2, \ldots, f_{m-1} \rbrace\] and \[\lbrace e_{m+1}, e_{m+2}, \ldots, e_{2m-1}, f_{m+1}, f_{m+2}, \ldots, f_{2m-1} \rbrace\] yields a commuting pair, so the regularity of $B_{\lambda}$ as a $\mathfrak{gl}_m \oplus \mathfrak{gl}_m$-crystal follows from its regularity as a $\mathfrak{gl}_{2m}$-crystal (interchanging $e_i$ and $f_i$ for all $1 \leq i \leq m-1$ interchanges $\Delta_i\epsilon_j$ with $\nabla_i \phi_j$ and $\Delta_i \phi_j$ with $\nabla_i \epsilon_j$ for all $1 \leq i,j \leq m-1$ --- cf. Stembridge \cite{Stemlocal}, p. 4809 --- so it does not affect the regularity of $B_{\lambda}$).  
\end{proof}

Each tableau in $B_{\lambda}$ is made up of two ``subtableaux'': a tableau with entries in $\lbrace 1, 2 \ldots, m \rbrace$ and a skew tableau with entries in $\lbrace m+1, m+2, \ldots, 2m \rbrace$.  These subtableaux do not interact with each other under any of the raising and lowering $(\mathfrak{gl}_m \oplus \mathfrak{gl}_m)$-crystal operators, so it is worthwhile to consider them independently.  

\begin{df}
For all tableaux $T \in B_{\lambda}$, let $\varphi_0(T)$ be the tableau obtained from $T$ by removing each box with an entry not in $\lbrace 1, 2, \ldots, m \rbrace$, and let $\varphi_1(T)$ be the skew tableau obtained from $T$ by removing each box with an entry not in $\lbrace m+1, m+2, \ldots, 2m \rbrace$, and reducing modulo $m$ the entry in each remaining box, so that the entries of $\varphi_1(T)$ are also among $1, 2, \ldots, m$.  Let $\varphi(T)$ be the ordered pair of tableaux $(\varphi_0(T), \operatorname{Rect}(\varphi_1(T)))$.  
\end{df}

\begin{prop} \label{ehighweights}
Let $\mathcal{C}$ be a connected component of the $(\mathfrak{gl}_m \oplus \mathfrak{gl}_m)$-crystal $B_{\lambda}$.  Then $\mathcal{C}$ is a highest weight crystal.  Furthermore, if $b$ is the unique highest weight element of $\mathcal{C}$, then there exist partitions $\beta = (\beta_1, \beta_2, \ldots, \beta_m)$ and $\gamma = (\gamma_1, \gamma_2, \ldots, \gamma_m)$ such that $b$ is of content $\overline{\beta} \gamma$ and $\varphi(b) = (b_{\overline{\beta}}, b_{\gamma})$, where $b_{\overline{\beta}}$ is the unique tableau of shape $\beta$ and content $\overline{\beta}$, and $b_{\gamma}$ is the unique tableau of shape $\gamma$ and content $\gamma$.  
\end{prop}

\begin{proof}
Since $\mathcal{C}$ is a regular, connected crystal, it follows from Proposition~\ref{regconequiv} that $\mathcal{C}$ is a highest weight crystal.  Let $b$ be the unique highest weight element of $\mathcal{C}$.  Recall that $b$ is anti-Yamanouchi in $\lbrace 1, 2, \ldots, m \rbrace$ and Yamanouchi in $\lbrace m+1, m+2, \ldots, 2m \rbrace$.  From Proposition~\ref{lowweightlemma}, we see that there exists a partition $\beta = (\beta_1, \beta_2, \ldots, \beta_m)$ such that $\varphi_0(b) = b_{\overline{\beta}}$, and, from Proposition~\ref{highweightlemma} (in view of Corollary~\ref{rectcommutes}), we see that there exists a partition $\gamma = (\gamma_1, \gamma_2, \ldots, \gamma_m)$ such that $\operatorname{Rect}(\varphi_1(T)) = b_{\gamma}$.  
\end{proof}

\begin{prop} \label{eprodcrystal}
Let $\beta$ and $\gamma$ be partitions, each with $m$ parts.  Equip the set $B_{(\overline{\beta}, \gamma)} := B_{\beta} \times B_{\gamma}$ with the map $\operatorname{wt} \times \operatorname{wt}$.  For all $1 \leq i \leq m-1$, let $e_i$ and $f_i$ act as the $\mathfrak{gl}_m$-crystal operators $e_i$ and $f_i$, respectively, on $B_{\beta}$ and as the identity on $B_{\gamma}$.  For all $m+1 \leq i \leq 2m-1$, let $e_i$ and $f_i$ act as the identity on $B_{\beta}$ and as the $\mathfrak{gl}_m$-crystal operators $e_{i-m}$ and $f_{i-m}$, respectively, on $B_{\gamma}$.  Then $B_{(\overline{\beta}, \gamma)}$, together with the set of raising operators $\lbrace f_1, f_2, \ldots, f_{m-1}, e_{m+1}, e_{m+2}, \ldots, e_{2m-1} \rbrace$ and the set of lowering operators $\lbrace e_1, e_2, \ldots, e_{m-1}, f_{m+1}, f_{m+2}, \ldots, f_{2m-1} \rbrace$, is a regular, connected $(\mathfrak{gl}_m \oplus \mathfrak{gl}_m)$-crystal with unique highest weight element $(b_{\overline{\beta}}, b_{\gamma})$.
\end{prop}

\begin{proof}
It is apparent that $B_{(\overline{\beta}, \gamma)}$ is a $(\mathfrak{gl}_m \oplus \mathfrak{gl}_m)$-crystal.  The regularity of $B_{(\overline{\beta}, \gamma)}$ follows from the regularity of $B_{\beta}$ and $B_{\gamma}$ as $\mathfrak{gl}_m$-crystals (as above interchanging $e_i$ and $f_i$ for all $1 \leq i \leq m-1$ interchanges $\Delta_i\epsilon_j$ with $\nabla_i \phi_j$ and $\Delta_i \phi_j$ with $\nabla_i \epsilon_j$ for all $1 \leq i,j \leq m-1$, so it does not affect the regularity of $B_{(\overline{\beta}, \gamma)}$).  

The claim that $(b_{\overline{\beta}}, b_{\gamma})$ is the unique highest weight element of $B_{(\overline{\beta},\gamma)}$ is a direct consequence of Propositions~\ref{lowweightlemma} and ~\ref{highweightlemma}.    
\end{proof}

Thus, if $\mathcal{C}$ is a connected component of $B_{\lambda}$, there exist partitions $\beta$ and $\gamma$ for which the unique highest weight element of $\mathcal{C}$ corresponds to that of $B_{(\overline{\beta}, \gamma)}$.  In fact, the two crystals are structurally identical.  

\begin{thm} \label{evachwcrystal}
Let $\mathcal{C}$ be a connected component of the $(\mathfrak{gl}_m \oplus \mathfrak{gl}_m)$-crystal $B_{\lambda}$.  Let $b$ be the unique highest weight element of $\mathcal{C}$, and let $\beta$ and $\gamma$ be partitions, each with $m$ parts, for which $\varphi(b) = (b_{\overline{\beta}}, b_{\gamma})$.  Then $\varphi$ restricts to an isomorphism of crystals $\mathcal{C} \xrightarrow{\sim} B_{(\overline{\beta}, \gamma)}$.  
\end{thm}

\begin{proof}
The content of $b$ is $\overline{\beta} \gamma$, so $\operatorname{wt}(b) = \operatorname{wt}(b_{\overline{\beta}}, b_{\gamma})$.  Furthermore, the equality $\phi_i(b) = \phi_i(b_{\overline{\beta}}, b_{\gamma})$ holds for all $1 \leq i \leq m-1$ by definition of $\varphi_0$, and it holds for all $m+1 \leq i \leq 2m-1$ by definition of $\varphi_1$ in view of Corollary~\ref{rectcommutes}.  Thus, Proposition~\ref{regconiso} tells us that $\mathcal{C}$ and $B_{(\overline{\beta}, \gamma)}$ are isomorphic.  

Since $\phi(b) = (b_{\overline{\beta}}, b_{\gamma})$ and \textit{jeu-de-taquin} slides commute with raising and lowering operators (cf. Proposition~\ref{jdtcommutes}), it follows that $\varphi |_{\mathcal{C}} \colon \mathcal{C} \rightarrow B_{(\overline{\beta}, \gamma)}$ is a morphism of crystals.  A morphism $\mathcal{C} \rightarrow B_{(\overline{\beta}, \gamma)}$ is uniquely determined by its image at $b$, so we may conclude that $\varphi |_{\mathcal{C}}$ is an isomorphism.  
\end{proof}

We turn our attention now to the action of $\xi$, first on the highest weight elements of the $(\mathfrak{gl}_m \oplus \mathfrak{gl}_m)$-crystal $B_{\lambda}$, and then on all its tableaux.  

\begin{lem} \label{evachighweight}
Let $b$ be a highest weight element of $B_{\lambda}$, and let $\beta$ and $\gamma$ be partitions, each with $m$ parts, such that $\varphi(b) = (b_{\overline{\beta}}, b_{\gamma})$.  Then $\varphi(\xi(b)) = (b_{\overline{\gamma}}, b_{\beta})$.  
\end{lem}

\begin{rem} \label{evacdescends}
If $\mu$ is a partition with $m$ parts, then $\textnormal{EYTab}(\lambda, \overline{\mu}\mu)$ is the set of highest weight elements of $B_{\lambda}$ with content $\overline{\mu}{\mu}$.  Thus, Lemma~\ref{evachighweight} implies that the Sch\"utzenberger involution indeed restricts to an action on $\textnormal{EYTab}(\lambda, \overline{\mu}\mu)$, as required for Theorem~\ref{mainevac} to be well-formulated.  
\end{rem}

\begin{proof}
In view of Proposition~\ref{evacinteracts}, we see that $\xi(b)$ is a highest weight element of $B_{\lambda}$ with content $\overline{\gamma}\beta$.  The desired result then follows directly from Proposition~\ref{ehighweights}.  
\end{proof}

\begin{thm} \label{evacgeneral}
Let $T \in B_{\lambda}$.  Then $\varphi(\xi(T)) = (\xi (\operatorname{Rect}(\varphi_1(T))), \xi (\varphi_0(T)))$.  
\end{thm}

\begin{rem}
To interpret the statement of Theorem~\ref{evacgeneral}, we understand $\xi$ to denote the Sch\"utzenberger involution on $\mathfrak{gl}_m$-crystals as well as that on $\mathfrak{gl}_{2m}$-crystals.  
\end{rem}

\begin{proof}
Let $T \in B_{\lambda}$, and let $\mathcal{C}$ be the connected component of $B_{\lambda}$ containing $T$.  Let $b$ be the unique highest weight element of $\mathcal{C}$.  Our proof is by induction on the length of the shortest path in the crystal from $b$ to $T$.  We see from Lemma~\ref{evachighweight} that the desired equality holds for the base case $T=b$.  

For the inductive step, it suffices to show that if $1 \leq i \leq m-1$ or $m+1 \leq i \leq 2m-1$, then \[\varphi(\xi(T)) = (\xi (\operatorname{Rect}(\varphi_1(T))), \xi (\varphi_0(T)))\] implies \[\varphi(\xi(f_i T)) = (\xi (\operatorname{Rect}(\varphi_1(f_i T))), \xi (\varphi_0(f_i T))).\]  Note that
\begin{align*}
\varphi(\xi(f_i T)) & = \varphi(e_{2m-i} \xi(T)) = e_{2m-i} \varphi(\xi(T)) \\ & = e_{2m-i}(\xi(\operatorname{Rect}(\varphi_1(T))), \xi(\varphi_0(T))).  
\end{align*}
If $1 \leq i \leq m-1$, then
\begin{align*}
\varphi(\xi(f_i T)) & = (\xi(\operatorname{Rect}(\varphi_1(T))), e_{m-i} \xi(\varphi_0(T))) \\ & = (\xi(\operatorname{Rect}(\varphi_1(T))), \xi(f_i \varphi_0(T))) \\ & = (\xi(\operatorname{Rect}(\varphi_1(f_i T))), \xi(\varphi_0(f_i T))).
\end{align*}
If $m+1 \leq i \leq 2m-1$, then
\begin{align*}
\varphi(\xi(f_i T)) & = (e_{2m-i} \xi(\operatorname{Rect}(\varphi_1(T))), \xi(\varphi_0(T))) \\ & = (\xi(f_{i-m}(\operatorname{Rect}(\varphi_1(T)))), \xi(\varphi_0(T))) \\ & = (\xi(\operatorname{Rect}(\varphi_1(f_i T))), \xi(\varphi_0(f_i T))).  
\end{align*}
\end{proof}

\begin{cor} \label{evacconncompresbiject}
Let $\mathcal{C}$ be a connected component of the $(\mathfrak{gl}_m \oplus \mathfrak{gl}_m)$-crystal $B_{\lambda}$, and let $b$ be the unique highest weight element of $\mathcal{C}$.  If $\xi(b) \neq b$, then $\lbrace T \in \mathcal{C} : \xi(T)=T \rbrace$ is empty.  Otherwise, there exists a partition $\beta = (\beta_1, \beta_2, \ldots, \beta_m)$ such that $\varphi(b) = (b_{\overline{\beta}}, b_{\beta})$, and the isomorphism of crystals $\varphi |_{\mathcal{C}} \colon \mathcal{C} \xrightarrow{\sim} B_{(\overline{\beta},\beta)}$ restricts to a bijection of sets \[\lbrace T \in \mathcal{C} : \xi(T) = T \rbrace \xrightarrow{\sim} \lbrace (U, U') \in B_{(\overline{\beta},\beta)} : \xi(U)=U' \rbrace.\]  
\end{cor}

We proceed to the proof of Theorem~\ref{mainevac} itself.  We compute the character $\chi$ of the $GL_{2m}(\mathbb{C})$-representation $V_{\lambda,2m}$ at the element \[w_0 \cdot \operatorname{diag}(x_1, x_2, \ldots, x_m, x_m, \ldots, x_2, x_1).\]  Note that \begin{align*} & \chi(w_0 \cdot \operatorname{diag}(x_1, x_2, \ldots, x_m, x_m, \ldots, x_2, x_1)) \\ & = (-1)^{v(\lambda)} \cdot \sum_{T \in B_{\lambda} : \xi(T) = T} x_1^{-2T_1} x_2^{-2T_2} \cdots x_m^{-2T_m} \\ & = (-1)^{v(\lambda)} \cdot \sum_{\theta \vdash |\lambda|/2} |\textnormal{EYTab}^{\xi}(\lambda, \overline{\theta} \theta)| \cdot s_{\theta}(x_1^{-2}, x_2^{-2}, \ldots, x_m^{-2}),\end{align*} where the first equality follows from Theorem~\ref{skanderabases}, and the second equality follows from Corollary~\ref{evacconncompresbiject}.  
  
However, \[w_0 \cdot \operatorname{diag}(x_1, x_2, \ldots, x_m, x_m, \ldots, x_2, x_1)\] is conjugate to \[\operatorname{diag}(x_1, x_2, \ldots, x_m, -x_m, \ldots, -x_2, -x_1).\]  Since $\chi \colon GL_{2m}(\mathbb{C}) \rightarrow \mathbb{C}$ is a class function, we see that \begin{align*} & \chi(w_0 \cdot \operatorname{diag}(x_1, x_2, \ldots, x_m, x_m, \ldots, x_2, x_1)) \\ & = \chi(\operatorname{diag}(x_1, x_2, \ldots, x_m, -x_m, \ldots, -x_2, -x_1)) \\ & = s_{\lambda}(x_1^{-1}, x_2^{-1}, \ldots, x_m^{-1}, -x_m^{-1}, \ldots, -x_2^{-1}, -x_1^{-1}) \\ & = (-1)^{v(\lambda)} \sum_D x_1^{-2D_1} x_2^{-2D_2} \cdots x_m^{-2D_m},\end{align*} where the sum ranges over all semistandard domino tableaux of shape $\lambda$ with entries in $\lbrace 1, 2, \ldots, m \rbrace$.  (Here the second equality follows from Theorem~\ref{skanderabases}, and the third equality follows from Remark 3.2 of Stembridge \cite{Stembridge}.)  

By Theorem~\ref{lascoux}, \[\sum_{D} x_1^{-2D_1} x_2^{-2D_2} \cdots x_m^{-2D_m} = \epsilon_2(\lambda) \cdot \phi_2(s_{\lambda})(x_1^{-2}, x_2^{-2}, \ldots, x_m^{-2}).\]  Expanding via Theorem~\ref{alphabetunique}, we find that \begin{align*} \phi_2(s_{\lambda})(x_1^{-2}, x_2^{-2}, \ldots, x_m^{-2}) & = \sum_{\theta \vdash |\lambda|/2} \langle \phi_2(s_{\lambda}), s_{\theta} \rangle s_{\theta}(x_1^{-2}, x_2^{-2}, \ldots, x_m^{-2}) \\ & = \sum_{\theta \vdash |\lambda|/2} \langle s_{\lambda}, p_2 \circ s_{\theta} \rangle s_{\theta}(x_1^{-2},x_2^{-2}, \ldots, x_m^{-2}).\end{align*}  Identifying the coefficients of $s_{\mu}(x_1^{-2}, x_2^{-2}, \ldots, x_m^{-2})$ in our two expressions for $\chi(w_0 \cdot \operatorname{diag}(x_1, x_2, \ldots, x_m, x_m, \ldots, x_2, x_1))$ in accordance with Theorem~\ref{alphabetunique}, we may conclude that \[ |\textnormal{EYTab}^{\xi}(\lambda, \overline{\mu} \mu)| =\epsilon_2(\lambda) \cdot \langle s_{\lambda}, p_2 \circ s_{\mu} \rangle.\] \qed

\subsection{Proof of Theorem~\ref{mainprom}}

Let $\lambda$ be a rectangular partition with $mn$ parts, and let $\mu$ be a partition of $|\lambda|/n$ with $m$ parts such that $\mu_i \leq \lambda_i$ for all positive parts $\mu_i$ of $\mu$.  Denote the $n$-fold concatenation \[(\mu_1, \mu_2, \ldots, \mu_m, \mu_1, \mu_2, \ldots, \mu_m, \ldots, \mu_1, \mu_2, \ldots, \mu_m)\] by $\mu^n$.  Write $\textnormal{Tab}(\lambda, \mu^n)$ for the set of semistandard tableaux of shape $\lambda$ and content $\mu^n$, and $\textnormal{PYTab}(\lambda, \mu^n)$ for the subset of $\textnormal{Tab}(\lambda, \mu^n)$ consisting of those tableaux with reading word Yamanouchi in $\lbrace km+1, km+2, \ldots, (k+1)m \rbrace$ for all $0 \leq k \leq n-1$.

\begin{rem}
Comments analogous to those in Remark~\ref{departure} apply here.
\end{rem}  

Let $B_{\lambda}$ be the set of semistandard tableaux of shape $\lambda$, endowed with a $\mathfrak{gl}_{mn}$-crystal structure in accordance with Proposition~\ref{tabstructure}.  The key to our proof is the assignment of a $(\mathfrak{gl}_m^{\oplus n})$-crystal structure to $B_{\lambda}$ that allows us to inspect the action of $j := \operatorname{pr}^m$ on its connected components.  This provides a combinatorial model for the decomposition into irreducible components of the restriction to $GL_m(\mathbb{C})^{\times n}$ of the irreducible $GL_{mn}(\mathbb{C})$-representation with highest weight $\lambda$, which underlies our character evaluation.  

Recall that we chose $\lbrace E_1 -E_2, E_2-E_3, \ldots, E_{mn-1} - E_{mn} \rbrace$ as the set of simple roots for $\mathfrak{gl}_{mn}$.  Here we choose \[\bigcup_{k=0}^{n-1} \lbrace E_{km+1} - E_{km+2}, E_{km+2} - E_{km+3}, \ldots, E_{(k+1)m-1} - E_{(k+1)m} \rbrace\] as the set of simple roots for $\mathfrak{gl}_m^{\oplus n}$.  

The following statements are analogous to statements 4.3 -- 4.7.  The proofs proceed exactly as in 4.3 -- 4.7 (if anything, they are even easier because in this case we do not flip the sign of any simple root).  Details may be found in Rush \cite{Rush} (be warned that there we work with $\mathfrak{sl}_m^{\oplus n}$-crystals rather than $\mathfrak{gl}_m^{\oplus n}$-crystals).   

\begin{prop} \label{promcrysrestrict}
The set $B_{\lambda}$ equipped with the map $\operatorname{wt}$, the set of raising operators $\bigcup_{k=0}^{n-1} \lbrace e_{km+1}, e_{km+2}, \ldots, e_{(k+1)m-1} \rbrace$, and the set of lowering operators $\bigcup_{k=0}^{n-1} \lbrace f_{km+1}, f_{km+2}, \ldots, f_{(k+1)m-1} \rbrace$ is a regular $(\mathfrak{gl}_m^{\oplus n})$-crystal. 
\end{prop}

\begin{df}
For all tableaux $T \in B_{\lambda}$ and $0 \leq k \leq n-1$, let $\varphi_k(T)$ be the tableau obtained from $T$ by removing each box with an entry not in $\lbrace km+1, km+2, \ldots, (k+1)m-1 \rbrace$, and reducing modulo $m$ the entry in each remaining box, so that the entries of $\varphi_k(T)$ are among $1, 2, \ldots, m$.  Let $\varphi(T)$ be the ordered $n$-tuple of tableaux $(\varphi_0(T), \operatorname{Rect}(\varphi_1(T)), \ldots, \operatorname{Rect}(\varphi_{n-1}(T)))$.  
\end{df}

\begin{prop} \label{phighweights}
Let $\mathcal{C}$ be a connected component of the $(\mathfrak{gl}_m^{\oplus n})$-crystal $B_{\lambda}$.  Then $\mathcal{C}$ is a highest weight crystal.  Furthermore, if $b$ is the unique highest weight element of $\mathcal{C}$, then there exist partitions $\beta_0, \beta_1, \ldots, \beta_{n-1}$, each with $m$ parts, such that $b$ is of content $\beta_0 \beta_1 \cdots \beta_{n-1}$ and \[\varphi(b) = (b_{\beta_0}, b_{\beta_1}, \ldots, b_{\beta_{n-1}}).\]  
\end{prop}

\begin{prop} \label{pprodcrystal}
Let $\beta_0, \beta_1, \ldots, \beta_{n-1}$ be partitions, each with $m$ parts.  Equip the set $B_{(\beta_0, \beta_1, \ldots, \beta_{n-1})} := B_{\beta_0} \times B_{\beta_1} \times \cdots \times B_{\beta_{n-1}}$ with the map $\operatorname{wt} \times \operatorname{wt} \times \cdots \times \operatorname{wt}$.  For all $1 \leq i \leq m-1$ and $0 \leq k \leq n-1$, let $e_{km+i}$ and $f_{km+i}$ act as the $\mathfrak{sl}_m$-crystal operators $e_i$ and $f_i$, respectively, on $B_{\beta_k}$ and as the identity on $B_{\beta_j}$ for all $j \neq k$.  Then $B_{(\beta_0, \beta_1, \ldots, \beta_{n-1})}$, together with the set of raising operators $\bigcup_{k=0}^{n-1} \lbrace e_{km+1}, e_{km+2}, \ldots, e_{(k+1)m-1} \rbrace$ and the set of lowering operators $\bigcup_{k=0}^{n-1} \lbrace f_{km+1}, f_{km+2}, \ldots, f_{(k+1)m-1} \rbrace$, is a regular, connected $(\mathfrak{gl}_m^{\oplus n})$-crystal with unique highest weight element $(b_{\beta_0}, b_{\beta_1}, \ldots, b_{\beta_{n-1}})$.
\end{prop}

\begin{thm} \label{promhwcrystal}
Let $\mathcal{C}$ be a connected component of the $(\mathfrak{gl}_m^{\oplus n})$-crystal $B_{\lambda}$.  Let $b$ be the unique highest weight element of $\mathcal{C}$, and let $\beta_0, \beta_1, \ldots, \beta_{n-1}$ be partitions, each with $m$ parts, for which $\varphi(b) = (b_{\beta_0}, b_{\beta_1}, \ldots, b_{\beta_{n-1}})$.  Then $\varphi$ restricts to an isomorphism of crystals $\mathcal{C} \xrightarrow{\sim} B_{(\beta_0, \beta_1, \ldots, \beta_{n-1})}$.
\end{thm}

We turn our attention now to the action of $j^d$ for $d$ dividing $n$, first on the highest weight elements of the $(\mathfrak{gl}_m^{\oplus n})$-crystal $B_{\lambda}$, and then on all its tableaux.  

\begin{lem} \label{promhighweight}
Let $b$ be a highest weight element of $B_{\lambda}$.  Let $\beta_0, \beta_1, \ldots, \beta_{n-1}$ be partitions, each with $m$ parts, such that $\varphi(b) = (b_{\beta_0}, b_{\beta_1}, \ldots, b_{\beta_{n-1}})$.  Then $\varphi(j^d(b)) = (b_{\beta_{n-d}}, b_{\beta_{n-d+1}}, \ldots, b_{\beta_{n-1}}, b_{\beta_0}, b_{\beta_1}, \ldots, b_{\beta_{n-d-1}})$.  
\end{lem}

\begin{rem} \label{promdescends}
If $\mu$ is a partition with $m$ parts, then $\textnormal{PYTab}(\lambda, \mu^n)$ is the set of highest weight elements of $B_{\lambda}$ with content $\mu^n$.  Thus, Lemma~\ref{promhighweight} implies that $j = \operatorname{pr}^m$ indeed restricts to an action on $\textnormal{PYTab}(\lambda, \mu^n)$, as required for Theorem~\ref{mainprom} to be well-formulated.  
\end{rem}

\begin{proof}
Recall from Theorem~\ref{uniqueprom} that $j^d(b) = j^{-(n-d)}(b)$.  In view of Proposition~\ref{prominteracts}, we see that $e_{km+1}, e_{km+2}, \ldots, e_{(k+1)m-1}$ all vanish at $j^d(b)$ for $d \leq k \leq n-1$.  Rewriting property (ii) in Proposition~\ref{prominteracts} as $\operatorname{pr}^{-1} (e_{i+1}(T)) = e_i (\operatorname{pr}^{-1}(T))$, we see that $e_{km+1}, e_{km+2}, \ldots, e_{(k+1)m-1}$ all vanish at $j^{-(n-d)}(b)$ for $0 \leq k \leq d-1$.  Thus, $j^d(b)$ is a highest weight element of $B_{\lambda}$, and, from property (i) in Proposition~\ref{prominteracts}, we see that $j^d(b)$ is of content \[b_{\beta_{n-d}} b_{\beta_{n-d+1}} \cdots b_{\beta_{n-1}} b_{\beta_0} b_{\beta_1} \cdots b_{\beta_{n-d-1}}.\]  The desired result then follows directly from Proposition~\ref{phighweights}.  
\end{proof}

\begin{thm} \label{promgeneral}
Let $T \in B_{\lambda}$.  Then $\varphi(j^d(T))$  \[= (\operatorname{Rect}(\varphi_{n-d}(T)), \ldots, \operatorname{Rect}(\varphi_{n-1}(T)), \varphi_0(T), \ldots, \operatorname{Rect}(\varphi_{n-d-1}(T))).\]  
\end{thm}

\begin{proof}
Let $T \in B_{\lambda}$, and let $\mathcal{C}$ be the connected component of $B_{\lambda}$ containing $T$.  Let $b$ be the unique highest weight element of $\mathcal{C}$.  Our proof is by induction on the length of the shortest path in the crystal from $b$ to $T$.  We see from Lemma~\ref{promhighweight} that the desired equality holds for the base case $T=b$.  

For the inductive step, it suffices to show that if $0 \leq k \leq n-1$ and $km+1 \leq i \leq (k+1)m-1$, then \[\varphi(j^d(T)) = (\operatorname{Rect}(\varphi_{n-d}(T)), \ldots, \varphi_0(T), \ldots, \operatorname{Rect}(\varphi_{n-d-1}(T)))\] implies \[ \varphi(j^d(f_i T)) = (\operatorname{Rect}(\varphi_{n-d}(f_i T)), \ldots, \varphi_0(f_i T), \ldots, \operatorname{Rect}(\varphi_{n-d-1}(f_i T))).\]  Note that if $0 \leq k \leq n-d-1$, then \begin{align*}  \varphi(j^d(f_i T)) & = \varphi (f_{i+dm} j^d(T)) = f_{i+dm} \varphi(j^d(T)) \\ & = f_{i+dm} (\operatorname{Rect}(\varphi_{n-d}(T)), \ldots, \varphi_0(T), \ldots, \operatorname{Rect}(\varphi_{n-d-1}(T))) \\ & = (\operatorname{Rect}(\varphi_{n-d}(f_i T)), \ldots, \varphi_0(f_i T), \ldots, \operatorname{Rect}(\varphi_{n-d-1}(f_i T))).\end{align*}
If $n-d \leq k \leq n-1$, then \begin{align*} \varphi(j^d(f_i T)) & = \varphi (f_{i-(n-d)m} j^d(T)) = f_{i-(n-d)m} \varphi(j^d(T)) \\ & = f_{i-(n-d)m} (\operatorname{Rect}(\varphi_{n-d}(T)), \ldots, \varphi_0(T), \ldots, \operatorname{Rect}(\varphi_{n-d-1}(T))) \\ & = (\operatorname{Rect}(\varphi_{n-d}(f_i T)), \ldots, \varphi_0(f_i T), \ldots, \operatorname{Rect}(\varphi_{n-d-1}(f_i T))). \end{align*}
\end{proof}

\begin{cor} \label{promconncompresbiject}
Let $\mathcal{C}$ be a connected component of the $(\mathfrak{gl}_m^{\oplus n})$-crystal $B_{\lambda}$, and let $b$ be the unique highest weight element of $\mathcal{C}$.  If $j^d(b) \neq b$, then $\lbrace T \in \mathcal{C} : j^d(T)=T \rbrace$ is empty.  Otherwise, there exist $d$ partitions $\beta_0, \beta_1, \ldots, \beta_d$, each with $m$ parts, such that $\varphi(b) = \left((b_{\beta_0}, b_{\beta_1}, \ldots, b_{\beta_{d-1}})^{n/d}\right)$, and the isomorphism of crystals $\mathcal{C} \xrightarrow{\sim} B_{\left((\beta_0, \beta_1, \ldots, \beta_d)^{n/d}\right)}$ restricts to a bijection of sets \begin{align*} & \lbrace T \in \mathcal{C} : j^d(T) = T \rbrace \xrightarrow{\sim} \\ & \lbrace (U_0, U_1, \ldots, U_{n-1}) \in B_{\left((\beta_0, \beta_1, \ldots, \beta_d)^{n/d}\right)} \\ & \text{such that } U_j = U_{j'} \text{ for all } j \cong j' \pmod d \rbrace. \end{align*}  
\end{cor}

We proceed to the proof of Theorem~\ref{mainprom} itself.  For all $T \in B_{\lambda}$, let \[(T_{1,1}, T_{1,2}, \ldots, T_{1,m}, T_{2,1}, T_{2,2}, \ldots, T_{2,m}, \ldots, T_{n,1}, T_{n,2}, \ldots, T_{n,m})\] be the content of $T$, and write $T_i$ for the composition $(T_{i,1}, T_{i,2}, \ldots, T_{i,m})$ for all $1 \leq i \leq n$.  Let $a$ be the number of positive parts of $\lambda$.  Let \[\lbrace y_{1,1}, y_{1,2}, \ldots, y_{1,m} \rbrace, \lbrace y_{2,1}, y_{2,2}, \ldots, y_{2,m} \rbrace, \ldots, \lbrace y_{d,1}, y_{d,2}, \ldots, y_{d,m} \rbrace\] be a collection of $d$ variable sets denoted by $y_1, y_2, \ldots, y_d$, respectively.  By abuse of notation, let the corresponding diagonal matrices \[\operatorname{diag}(y_{1,1}, y_{1,2}, \ldots, y_{1,m}), \operatorname{diag}(y_{2,1},y_{2,2}, \ldots, y_{2,m}), \ldots, \operatorname{diag}(y_{d,1},y_{d,2}, \ldots, y_{d,m})\] be denoted by $y_1, y_2, \ldots, y_d$, respectively, as well.  We compute the character $\chi$ of the $GL_{mn}(\mathbb{C})$-representation $V_{\lambda,mn}$ at the element \[c_{mn}^{md} \cdot \operatorname{diag}(y_1, y_2, \ldots, y_d, y_1, y_2, \ldots, y_d, \ldots, y_1, y_2, \ldots, y_d).\]

Let $\eta$ be a composition of $|\lambda|$ with $mn$ parts, and let $U$ be a semistandard tableau of shape $\lambda$ and content $\eta$.  From Theorem~\ref{skanderabases}, we see that \[c_{mn}^{md} \cdot I_{\eta}(U) = (-1)^{(\eta_{m(n-d)+1} + \eta_{m(n-d)+2} + \cdots + \eta_{mn})(a-1)} \cdot I_{c_{mn}^{md} \cdot \eta}(\operatorname{pr}^{md}(U)).\]  If $\eta = c_{mn}^{md} \cdot \eta$, then \[\eta_{m(n-d)+1} + \eta_{m(n-d)+2} + \cdots + \eta_{mn} = \frac{d}{n} \cdot |\eta| = \frac{d}{n} \cdot |\lambda|,\] and \[c_{mn}^{md} \cdot I_{\eta}(U) = (-1)^{\frac{d}{n} \cdot |\lambda| (a-1)} \cdot I_{\eta}(\operatorname{pr}^{md}(U)) = \zeta^{d v(\lambda)} \cdot I_{\eta}(\operatorname{pr}^{md}(U)),\] where $\zeta$ is a primitive $n^{\text{th}}$ root of unity.  

Note that \begin{align*} & \chi \left (c_{mn}^{md} \cdot \operatorname{diag}(y_1, y_2, \ldots, y_d, y_1, y_2, \ldots, y_d, \ldots, y_1, y_2, \ldots, y_d) \right) \\ & = \zeta^{d v(\lambda)} \cdot \sum_{T \in B_{\lambda} : j^d(T) = T} y_1^{-\frac{n}{d} \cdot T_1} y_2^{-\frac{n}{d} \cdot T_2} \cdots y_d^{-\frac{n}{d} \cdot T_d}\end{align*} \[= \zeta^{d v(\lambda)} \cdot \sum |\textnormal{PYTab}^{j^d}(\lambda, (\theta_1 \theta_2 \ldots \theta_d)^\frac{n}{d})| \cdot s_{\theta_1}\left(y_1^{-\frac{n}{d}}\right) s_{\theta_2}\left(y_2^{-\frac{n}{d}}\right) \cdots s_{\theta_d}\left(y_d^{-\frac{n}{d}}\right),\] where the sum ranges over all $d$-tuples of partitions $(\theta_1, \theta_2, \ldots, \theta_d)$ such that $|\theta_1| = |\theta_2| = \cdots = |\theta_d| = |\lambda|/n$.  (The first equality follows from Theorem~\ref{skanderabases}, and the second equality follows from Corollary~\ref{promconncompresbiject}.)
  
However, \[c_{mn}^{md} \cdot \operatorname{diag}(y_1, y_2, \ldots, y_d, y_1, y_2, \ldots, y_d, \ldots, y_1, y_2, \ldots, y_d)\] is conjugate to \[\operatorname{diag}(y_1, y_2, \ldots, y_d, \zeta^d y_1, \zeta^d y_2, \ldots, \zeta^d y_d, \ldots, \zeta^{n-d} y_1, \zeta^{n-d} y_2, \ldots, \zeta^{n-d} y_d).\]  Since $\chi \colon GL_{mn}(\mathbb{C}) \rightarrow \mathbb{C}$ is a class function, we see that \begin{align*} & \chi(c_{mn}^{md} \cdot \operatorname{diag}(y_1, y_2, \ldots, y_d, y_1, y_2, \ldots, y_d, \ldots, y_1, y_2, \ldots, y_d)) \\ & = \chi(\operatorname{diag}(y_1, y_2, \ldots, y_d, \zeta^d y_1, \zeta^d y_2, \ldots, \zeta^d y_d, \ldots, \zeta^{n-d} y_1, \zeta^{n-d} y_2, \ldots, \zeta^{n-d} y_d)) \\ & = s_{\lambda}(y_1^{-1},  \ldots, y_d^{-1}, \zeta^d y_1^{-1}, \ldots, \zeta^d y_d^{-1}, \ldots, \zeta^{n-d} y_1^{-1}, \ldots, \zeta^{n-d} y_d^{-1}) \\ & = \zeta^{d v(\lambda)} \sum_R y_1^{-\frac{n}{d} \cdot R_1} y_2^{-\frac{n}{d} \cdot R_2} \cdots y_d^{-\frac{n}{d} \cdot R_d},\end{align*} where the sum ranges over all semistandard $\frac{n}{d}$-ribbon tableaux of shape $\lambda$ with entries in $\lbrace 1, 2, \ldots, md \rbrace$.  (For all such $\frac{n}{d}$-ribbon tableaux $R$, the content of $R$ is denoted by \[(R_{1,1}, R_{1,2}, \ldots, R_{1,m}, R_{2,1}, R_{2,2}, \ldots, R_{2,m}, \ldots, R_{d,1}, R_{d,2}, \ldots, R_{d,m}),\] and, for all $1 \leq i \leq d$, the composition $(R_{i,1}, R_{i,2}, \ldots, R_{i,m})$ is denoted by $R_i$.)  Here the second equality follows from Theorem~\ref{skanderabases}, and the third equality follows from Lemma 6.2 of Rhoades \cite{Rhoades}.  

By Theorem~\ref{lascoux}, \[\sum_{R} y_1^{-\frac{n}{d} \cdot R_1} y_2^{-\frac{n}{d} \cdot R_2} \cdots y_d^{-\frac{n}{d} \cdot R_d} = \epsilon_{n/d}(\lambda) \cdot \phi_{n/d}(s_{\lambda})\left(y_1^{-\frac{n}{d}}, y_2^{-\frac{n}{d}}, \ldots, y_d^{-\frac{n}{d}}\right).\]  Expanding via Theorem~\ref{alphabetunique}, we find that \begin{align*} & \phi_{n/d}(s_{\lambda})\left(y_1^{-\frac{n}{d}}, y_2^{-\frac{n}{d}}, \ldots, y_d^{-\frac{n}{d}}\right) \\ & = \sum \langle \phi_{n/d}(s_{\lambda}), s_{\theta_1} s_{\theta_2} \cdots s_{\theta_d} \rangle s_{\theta_1}\left(y_1^{-\frac{n}{d}}\right) s_{\theta_2}\left(y_2^{-\frac{n}{d}}\right) \cdots s_{\theta_d}\left(y_d^{-\frac{n}{d}}\right) \\ & = \sum \langle s_{\lambda}, p_{n/d} \circ (s_{\theta_1} s_{\theta_2} \cdots s_{\theta_d}) \rangle s_{\theta_1}\left(y_1^{-\frac{n}{d}}\right) s_{\theta_2}\left(y_2^{-\frac{n}{d}}\right) \cdots s_{\theta_d}\left(y_d^{-\frac{n}{d}}\right),\end{align*} where again the sums range over all $d$-tuples of partitions $(\theta_1, \theta_2, \ldots, \theta_d)$ such that $|\theta_1| = |\theta_2| = \cdots = |\theta_d| = |\lambda|/n$. 

Note that $p_{n/d} \circ (s_{\theta_1} s_{\theta_2} \cdots s_{\theta_d}) = (s_{\theta_1} s_{\theta_2} \cdots s_{\theta_d}) \circ p_{n/d}$ in view of Proposition~\ref{plethyswitch}.  It follows from Equation 6.4 in Macdonald \cite{Macdonald}, Chapter 1, that $(g_1 g_2) \circ h = (g_1 \circ h)(g_2 \circ h)$ for all symmetric functions $g_1, g_2, h \in \Lambda$.  Thus, we see inductively that \[(s_{\theta_1} s_{\theta_2} \cdots s_{\theta_d}) \circ p_{n/d} = (s_{\theta_1} \circ p_{n/d})(s_{\theta_2} \circ p_{n/d}) \cdots (s_{\theta_d} \circ p_{n/d}).\]  Invoking Proposition~\ref{plethyswitch} again, we find that \[p_{n/d} \circ (s_{\theta_1} s_{\theta_2} \cdots s_{\theta_d}) = (p_{n/d} \circ s_{\theta_1})(p_{n/d} \circ s_{\theta_2}) \cdots (p_{n/d} \circ s_{\theta_d}).\]  

Thus, identifying the coefficients of $s_{\mu}(y_1^{-\frac{n}{d}}) s_{\mu}(y_2^{-\frac{n}{d}}) \cdots s_{\mu}(y_d^{-\frac{n}{d}})$ in our two expressions for \[\chi(c_{mn}^{md} \cdot \operatorname{diag}(y_1, y_2, \ldots, y_d, y_1, y_2, \ldots, y_d, \ldots, y_1, y_2, \ldots, y_d))\] in accordance with Theorem~\ref{alphabetunique}, we may conclude that \[|\textnormal{PYTab}^{j^d}(\lambda, \mu^n)| = \epsilon_{n/d}(\lambda) \cdot \langle s_{\lambda}, p_{n/d}^d \circ s_{\mu} \rangle.\] \qed

\section{Acknowledgments}
This research was undertaken at the University of Michigan, Ann Arbor, under the direction of Prof. David Speyer and with the financial support of the US National Science Foundation via grant DMS-1006294.  It is the author's pleasure to extend his gratitude first and foremost to Prof. Speyer for his dedicated mentorship while this project was in progress and his continued support when it came time to ready the results for eventual publication.  The author would also like to thank Prof. Michael Zieve for his leadership of the REU (Research Experiences for Undergraduates) program hosted by the University of Michigan.  He thanks Victor Reiner, Brendon Rhoades, Richard Stanley, and John Stembridge for helpful conversations.  Finally, he thanks Daniel Bump and Travis Scrimshaw for numerous editorial suggestions.  

The author is presently supported by the NSF Graduate Research Fellowship Program.

\end{document}